\documentclass[a4paper]{article}

\pdfoutput=1

\usepackage{amssymb, times, amsmath, amsthm, extpfeil}

\usepackage{url}
\usepackage{cancel}

\theoremstyle{plain}
\newtheorem{theorem}{Theorem}[section]
\newtheorem{lemma}[theorem]{Lemma}
\newtheorem{proposition}[theorem]{Proposition}
\newtheorem{corollary}[theorem]{Corollary}

\theoremstyle{definition}

\newtheorem{definition}[theorem]{Definition}

\theoremstyle{remark}
\newtheorem*{remark}{Remark}
\newtheorem*{notation}{Notation}

\makeatletter
\def\imod#1{\allowbreak\mkern10mu({\operator@font mod}\,\,#1)}
\makeatother

\def\rca#1{\sf{RCA}_{#1}}
\def\rpea#1{\sf{RPEA}_{#1}}
\def\pea#1{\sf{PEA}_{#1}}
\def\pa#1{\sf{PA}_{#1}}
\def\rpa#1{\sf{RPA}_{#1}}
\def\rdf#1{\sf{RDf}_{#1}}
\def\ca#1{\sf{CA}_{#1}}
\def\df#1{\sf{Df}_{#1}}
\def\rl{\sf{RL}}
\def\restr{\mathop{\restriction}}
\def\At{\mathop{At}}

\def\dom{\mathop{dom}}
\def\im{\mathop{im}}

\def\set{\im}

\let\sec\S
\def\A{\mathcal{A}}
\def\B{\mathcal{B}}
\def\C{\mathcal{C}}

\def\G{\mathcal{G}}
\def\M{\mathcal{M}}
\def\N{\mathcal{N}}
\def\P{\mathcal{P}}
\def\S{\mathcal{S}}
\def\U{\mathcal{U}}

\def\c #1{\mathcal{#1}}
\def\vec #1#2{#1_1,\allowbreak\ldots,\allowbreak #1_{#2}}
\def\ds #1{$#1$-distin\-guish\-ing}
\def\osim{{\sim}}
\def\nb #1{$\bullet$\marginpar{\small #1}}
\def\nb #1{}

\title{Bare canonicity of representable cylindric and polyadic algebras}

\author{Jannis Bulian$^1$ and Ian Hodkinson$^2$}

\date{ \small
${}^1$ Mathematical Institute, University of Oxford, 24--29 St Giles', Oxford OX1 3LB, UK
\\[6pt]
${}^2$ Department of Computing, Imperial College London, London SW7 2AZ, UK
\\
\url{http://www.doc.ic.ac.uk/~imh/}}

\begin{document}

\maketitle

\begin{abstract}
We show that for finite $n \geq 3$, every first-order
axiomatisation of the varieties
of representable $n$-dimensional cylindric algebras,
diagonal-free cylindric algebras, poly\-adic algebras, and polyadic equality algebras
contains an infinite number of non-canonical formulas.
We also show that the class of structures for each of these varieties is non-elementary.
The proofs employ algebras derived from random graphs.
\end{abstract}

\section{Introduction}
The notion of the \emph{canonical extension} of a boolean algebra with operators
(or `BAO') was introduced by J\'onsson and Tarski in a classical paper \cite{Jonsson51},
generalising a construction of Stone \cite{Stone36}.
It is an algebra whose domain is the power set of the set of ultrafilters
of the original BAO,
and its operations are induced from those of the BAO in a natural way.
Canonical extensions are nowadays a key tool in algebraic logic,
with a multitude of uses and generalisations.

A class of BAOs is said to be \emph{canonical} if it is closed under taking canonical extensions.
In this paper we are concerned with the 
classes of representable $n$-dimensional cylindric algebras,
diagonal-free cylindric algebras, polyadic algebras,
and polyadic equality algebras, for finite $n\geq3$.
These four classes are varieties.
They are non-fin\-itely axiomatisable,
and many further `negative' results on axiomatisations are known
(e.g., \cite{Andreka97,Ahmed11}).
However, the classes  are canonical.
Now \cite{Jonsson51} already established that positive equations
are preserved by canonical extensions,
and more generally, Sahlqvist equations are also preserved (see, e.g., \cite{Blackburn01}).
This may suggest that the four classes might be axiomatisable by positive or
Sahlqvist equations.

It turned out that the representable
cylindric algebras are not Sahlqvist axiomatisable  \cite[footnote 1]{V97:atomSahl}.
In this paper, we extend this result to a wider class of axioms and
to all four classes.
A  first-order sentence is said to be \emph{canonical} if the class of its BAO models is canonical. 
Although some syntactic classes of canonical sentences
(such as Sahlqvist equations) are known,
canonicity is a semantic property that cannot be easily defined 
syntactically.  For example, there is no algorithm to decide whether an equation is canonical
\cite[Theorem 9.6.1]{Kracht99}.
The goal of this paper is to show that there is no canonical axiomatisation of
any of the four classes listed above. In fact,
we will show that \emph{any first-order axiomatisation of any of them contains infinitely many
non-canonical sentences.}
We say that a canonical class of BAOs with this 
property is \emph{barely canonical.}
Although the class is canonical, its canonicity
emerges only `in the limit' and does not reside in any finite number of axioms
for it, however they are phrased.

There are a few related results in the literature.
The class of representable relation algebras,
proved to be canonical  by Monk 
(reported in \cite{McKe66}), was shown in \cite{Hodkinson05} to be barely canonical.
Our proof  in the current paper is similar but somewhat simpler: the use of finite combinatorics 
(finite Ramsey theorem, etc) in \cite{Hodkinson05} is replaced here by the use of  first-order compactness. Bare canonicity of the `McKinsey--Lemmon' modal logic
was shown in \cite{Goldblatt07}.

We sketch the rough outline of the proof.
Our aim is to convey the idea quickly, and the description will not be completely accurate in detail.
Our construction uses polyadic-type  algebras built from graphs.
They are polyadic expansions of  cylindric-type algebras constructed from graphs 
in~\cite{Hirsch09}, where it was shown (roughly) that
 such an algebra is representable if and only if its base graph has
infinite chromatic number.
(This was used in~\cite{Hirsch09} to prove that the class of structures
for the variety of representable $n$-dimensional cylindric algebras
(finite $n\geq3$) is non-elementary, a result generalised
to diagonal-free, polyadic, and polyadic equality algebras in Theorem~\ref{thm:srb nonelem} below.)
Here, we will cast this work in a wider setting by defining an elementary class $\sf K$ of 
three-sorted structures comprising a polyadic equality-type algebra $\A$, a graph $\G$, and a boolean algebra $\B$ of subsets of $\G$.
We will show that  representability of $\A$
is  equivalent to $\G$
having infinite chromatic number in the sense of $\B$.
Both these properties can be defined by first-order theories,
which therefore have the same models modulo the theory defining $\sf K$.
It follows by compactness that
 if the class of representable algebras had a first-order axiomatisation
using only canonical sentences, 
there would be a function $f:\omega\to\omega$ such that
whenever an algebra $\A$  has chromatic number at least $f(k)$
(in the sense of some three-sorted structure), its canonical extension has
chromatic number at least $k$.
We then borrow from \cite{Hodkinson05} an inverse system of finite (random)
graphs of chromatic number $m$ whose inverse limit has chromatic number $k$,
for any chosen $2\leq k<m<\omega$. Using some results 
of Goldblatt \cite{Goldblatt76} connecting canonical extensions with inverse limits,
this yields an algebra of chromatic number $m$ whose canonical extension has chromatic number $k$.
Since $k,m$ are arbitrary, no function $f$ as above can exist.
A slight extension of the argument, using a little more compactness, 
shows that
any first-order axiomatisation of the representable algebras  has 
infinitely many non-canonical sentences.

\paragraph{Layout of paper}
In Section~\ref{sec:algs of rels} we recall some basic notions of
algebras of relations, representability, duality and canonicity.
We define polyadic equality-type algebras over graphs in Section~\ref{S:alggraphs},
and  abstract generalisations of them in Section~\ref{sec:ags},
where we also ascertain some of their elementary properties.
This is continued in Section~\ref{S:netpatchsys}, where we study their ultrafilters. 
In Section~\ref{sec:nks} we introduce
approximations to representations by means of systems of ultrafilters called `ultrafilter networks',
and lower-dimensional approximations of them called `patch systems'.
This will allow us to prove in
Section~\ref{S:chromrep} that (roughly) an abstract algebra is representable
if and only if its associated graph has infinite chromatic number. Assuming
an axiomatisation with only finitely many non-canonical formulas, we use direct
and inverse systems in Section~\ref{S:dirinvsys} to build an algebra that satisfies an arbitrary number of axioms,
while its canonical extension satisfies only a bounded number, and thus obtain a
contradiction.
Section~\ref{sec:end} lists some open problems.

\paragraph{Notation}
We use the following notational conventions. 
We usually identify (notationally) a structure, algebra, or graph
with its domain.
For signatures $L\subseteq L'$ and an $L'$-structure $M$,
we write $M\restr L$ for the $L$-reduct of~$M$.

Throughout the paper, the dimension $n$ is a
fixed finite positive integer and $n$ is at least $3$. 
It will often be implicit that cylindric algebras etc.\ are $n$-dimen\-sional
and that $i,j,k,m$, etc., denote indices $<n$. We identify a non-negative
integer $m$ with the set $\{0,1, \dots, m-1\}$.
If $V$ is a set, we write $[V]^m$ for the set of subsets of size $m$ of V. We write
$\omega$ for the first infinite ordinal number. 
$\wp(S)$ denotes the power set of a set $S$.

For a function $f:X\to Y$ we write $\dom f$ for its domain, $\im f$ for its image,
and $f[X']$ for $\{f(x')\mid x'\in X'\}$ when $X'\subseteq X$.
We use similar notation for $m$-ary functions, for $m<\omega$
--- e.g., in Definition~\ref{def:uf str}.
For functions $f,g$, we write $f\circ g$ for their composition: $f\circ g(x)=f(g(x))$.
We omit  brackets in
function applications when we believe it improves readability.
By ${}^\alpha U$, where $\alpha$ is an ordinal,
we denote the set of functions from $\alpha$ to $U$, so an $\alpha$-ary
relation on $U$ is a subset of ${}^\alpha U$. To keep the syntax similar to the finite
case, we write $x_i$ for $x(i)$ if $x \in {}^\alpha U$ and $i < \alpha$. 

\section{Algebras of relations}\label{sec:algs of rels}
In this paper, we will consider four types of algebra:  cylindric-type algebras,
diagonal-free cylindric-type algebras, polyadic-type algebras, and polyadic equality-type algebras,
all of dimension $n$. They differ in their signatures and  notion of representation.
Here, we define them formally and recall some aspects of duality and canonicity for them.

\subsection{Signatures and algebras}

\begin{definition}\label{def:sigs}
We let
\begin{enumerate}

\item $L_{BA}=\{+,-,0,1\}$ denote the signature of boolean algebras,

\item $L_{\ca n}=L_{BA}\cup\{c_i,d_{ij}\mid i,j<n\}$
denote the signature
of $n$-dimensional  cylindric  algebras,

\item $L_{\df n}=L_{BA}\cup\{c_i\mid i<n\}$,
denote the signature
of $n$-dimensional diagonal-free cylindric  algebras,

\item $L_{\pa n}=L_{BA}\cup\{c_i,s_\sigma\mid i,j<n,\;\sigma:n\to n\}$
denote the signature
of $n$-dimensional polyadic  algebras,

\item $L_{\pea n}=L_{BA}\cup\{c_i,d_{ij},s_\sigma\mid i,j<n,\;\sigma:n\to n\}$
denote the signature
of $n$-dimensional polyadic equality algebras.

\end{enumerate}
Here,  the $c_i$ (`cylindrifications') and $s_\sigma$
(`substitutions')  are unary function symbols and the $d_{ij}$ (`diagonals') are constants.
By a \emph{cylindric-type algebra,} we mean simply an algebra of signature $L_{\ca n}$.
\emph{Diagonal-free cylindric-type algebras, polyadic-type algebras,} 
and \emph{polyadic equality-type algebras} are defined analogously for the other signatures.
\end{definition}

Our concern in this paper is with representable algebras of these four kinds,
but we briefly note that abstract
algebras have been defined as well:
namely, cylindric  algebras, diagonal-free cylindric  algebras,
polyadic  algebras, and polyadic equality algebras. 
They
are algebras of the above types that satisfy
in each case a finite set of equations that can be found in \cite{Henkin71,Henkin85}.
In particular, \emph{cylindric algebras} are defined in \cite[Definition 1.1.1]{Henkin71}.
We will not use the formal definition so we do not recall it here,
but the proofs of some later lemmas will be easier for readers familiar with
basic computations in cylindric algebras.  The material in
\cite[\sec1]{Henkin71} is easily enough for what we need.
Readers not so familiar can easily verify our claims directly
in the specific algebras we are working with.

\subsection{Representations}
Natural examples of each kind of algebra arise
from algebras of $n$-ary relations on a set.

\begin{definition}
\index{cylindric set algebra}
A \emph{polyadic equality set algebra}
is a polyadic equality-type algebra of the form
\[
(\wp(V), \emptyset, V, \cup, V\setminus-, C^U_i, D^U_{ij},S^U_\sigma\mid i,j < n),
\]
where  $U$ is a non-empty set, $V = {}^n U$, and
\begin{enumerate}
\item $D^U_{ij} = \{x \in V \mid x_i = x_j\}$ for $i,j < n$,
\item $C^U_i X = \{x \in V \mid \exists y \in X \;\forall j < n (j \neq i \rightarrow y_j = x_j)\}$ for $i < n$ and $X \subseteq V$,

\item $S^U_\sigma X=\{x\in V\mid x\circ\sigma\in X\}$, for $\sigma:n\to n$ and $X\subseteq V$.
\end{enumerate}

A \emph{polyadic set algebra} (\emph{cylindric set algebra})
is the reduct of a polyadic equality set algebra to the signature $L_{\pa n}$
(respectively, $L_{\ca n}$). 
Since $L_{\df n}$ has no operations connecting two different dimensions,
a \emph{diagonal-free cylindric set algebra}
is defined rather differently, as an $L_{\df n}$-algebra of the form
\[
(\wp(V), \emptyset, V, \cup, \setminus, C^V_i \mid i < n),
\]
where  $U_0,\ldots,U_{n-1}\neq\emptyset$,
 $V = \prod_{i<n}U_i$, and  $C^V_{i}X = \{x \in V \mid \exists y \in X \;\forall j \in n\setminus\{i\} ( y_j = x_j)\}$ for $i < n$ and $X\subseteq V$.  
\end{definition}

\begin{definition}

An $L_{\pea n}$-algebra is said to be \emph{representable} if it is isomorphic to a
subalgebra of a product of polyadic equality set algebras. The isomorphism is then called
a \emph{representation}.
The class of all representable polyadic equality algebras of dimension $n$ is
called $\rpea{n}$.

Exactly analogous definitions are made for $L_{\df n}$,
$L_{\ca n}$, and $L_{\pa n}$, 
using the appropriate set algebras in each case.
The classes of representable algebras for these are,
respectively,
 $\rdf n$, $\rca n$, and $\rpa n$.
\end{definition}

It is known that 
$\rpea n$, $\rpa n$, $\rca n$, and $\rdf n$ are varieties
(elementary classes defined by equations):
see, e.g., \cite[3.1.108, 5.1.43]{Henkin85}.
For $n\geq 3$, they are not finitely axiomatisable
\cite{Monk69,Johnson69},
and indeed we will see this later in Corollary~\ref{cor:nfa etc}.

\subsection{Atom structures}\label{ss:as}
We now recall a little duality theory, leading to canonicity, the 
topic of the paper.  
For more details, see, e.g., \cite{Jonsson51,Blackburn01} and \cite[\sec2.7]{Henkin71}.

\begin{definition}\label{def:as}
Let $L\supseteq L_{BA}$ be a functional signature (i.e., one with only function symbols and constants).
We write $L_+$ for the relational signature
consisting of an $(n+1)$-ary relation symbol $R_f$ for each $n$-ary 
function symbol $f\in L\setminus L_{BA}$.
By an ($L$)-\emph{atom structure,} we will simply mean an $L_+$-structure.
We will sometimes refer to the elements of an atom structure as \emph{atoms.}

Given an $L$-atom structure $\S=(S,R_f\mid R_f\in L_+)$,
we write $\S^+$ for its \emph{complex algebra:}
$\S^+=(\wp(S),f\mid f\in L)$,
where each $f\in L_{BA}$ is interpreted 
in the natural way as a boolean operation on $\wp(S)$,
and $f(\vec Xn)=\{s\in S\mid \S\models R_f(\vec xn,s)$ for some $x_1\in X_1,\ldots,x_n\in X_n\}$,
for each  $n$-ary $f\in L \setminus L_{BA}$ and $\vec Xn\subseteq S$.
We identify each $s\in S$ with the atom $\{s\}$ of $\S^+$.
\end{definition}
 
For the particular signature $L_{\pea n}$, 
we will be defining atom structures in which $R_{c_i}$ is an equivalence relation
and $R_{s_\sigma}^{-1}$  a function.  So we 
adopt a slightly different definition of atom structure
that is a little easier to specify in practice.

\begin{definition}\label{def:polyas}
A \emph{polyadic equality atom structure} is a structure
\[
\S=(S,D_{ij},\equiv_i,-^\sigma\mid i,j<n,\;\sigma:n\to n),
\]
where $D_{ij}\subseteq S$, $\equiv_i$ is an equivalence relation on $S$,
and $-^\sigma:S\to S$ is a function.
We regard $\S$ as a standard $L_{\pea n}$-atom structure
in the sense of Definition~\ref{def:as} by interpreting
$R_{d_{ij}}$ as $D_{ij}$, $R_{c_i}$ as $\equiv_i$,
and letting $R_{s_\sigma}(s,t)$ iff $t^\sigma=s$.
\end{definition}

\subsection{Canonicity}
One source of atom structures is from \emph{boolean algebra with operators (BAOs).}
These
originated in \cite{Jonsson51},
where they were called `normal BAOs', and they are now familiar: see, e.g., \cite{Blackburn01}. Let $L$ be a functional signature containing $L_{BA}$.

\begin{definition} \label{D:ufstr}
\index{ultrafilter structure}
An \emph{$L$-BAO} is an $L$-structure whose $L_{BA}$-reduct is a boolean algebra
and in which each $f\in L\setminus L_{BA}$ defines
a function that is normal
(its value is zero whenever any argument is zero) and additive in each argument.
\end{definition}
For example, if $\S$ is an $L$-atom structure
 then $\S^+$ is an $L$-BAO (note that constants are vacuously normal and additive).
Any algebra in $\rdf n$, $\rca n$, $\rpa n$, and $\rpea n$ is easily checked to be a BAO for its signature.

\begin{definition}\label{def:uf str}
Let $\B$ be an $L$-BAO. We
define the \emph{ultrafilter structure}
$\B_+$ to be the $L$-atom structure which has the set of ultrafilters (of the boolean
reduct) of  $\B$ as domain and,
for any $n$-ary function symbol $f \in L\setminus L_{BA}$ and $\mu_0, \dots, \mu_{n-1}, \nu \in \B_+$,
\[
\B_+ \models R_f(\mu_0, \dots, \mu_{n-1}, \nu)
  \iff f[\mu_0, \dots, \mu_{n-1}] \subseteq \nu.
\]
The \emph{canonical extension}  of $\B$, denoted by $\B^\sigma$,
is the $L$-BAO $(\B_+)^+$.
A~class $\sf K$ of $L$-BAOs is said to be \emph{canonical}
if $\B\in\sf K$ implies $\B^\sigma\in\sf K$.
A first-order $L$-sentence $\theta$ is said to be \emph{canonical}
if $\B\models\theta$ implies $\B^\sigma\models\theta$ for every $L$-BAO $\B$.
\end{definition}
Canonical extensions were introduced in \cite{Jonsson51}, where it was shown
that there is a canonical embedding of $\B$ into $\B^\sigma$
given by $b\mapsto\{\nu\in\B_+\mid b\in\nu\}$,
so justifying the use of `extension'.
Canonical extensions of cylindric algebras are 
studied in \cite[\sec2.7]{Henkin71}.
Canonical varieties in general have been intensely studied, 
for example by Goldblatt \cite{Gol95:canon},
and it is not hard to derive the following well known result.
The proof we give follows \cite{Gol95:canon}:
\cite[Theorem 4.6]{Gol95:canon} proves by a stronger version of the same method
that $\rca\alpha$ and ${\bf I}{\sf Crs}_\alpha$
are canonical varieties for every ordinal $\alpha$. Canonicity of $\rca n$ is proved 
in a different way in  \cite[p.459]{Henkin71}.

\begin{proposition}\label{prop:can}
$\rdf n$, $\rca n$, $\rpa n$, and $\rpea n$ are canonical varieties.
\end{proposition}

\begin{proof}
Let 
$\sf K_{Df}$ be the class of all $(L_{\df n})_+$-structures 
of the form $(\prod_{i<n}U_i,R_{c_i}\mid i<n)$,
where $U_0,\ldots,U_{n-1}\neq\emptyset$ and $R_{c_i}((u_0,\ldots,u_{n-1}),(v_0,\ldots,v_{n-1}))$
iff $u_j=v_j$ for each $j\in n\setminus\{i\}$.
Let $\sf K_{PEA}$ be  the class of $(L_{\pea n})_+$-structures of the form
$(^nU,R_{c_i},R_{d_{ij}},R_{s_\sigma}\mid i,j<n,\;\sigma:n\to n)$,
where $U\neq\emptyset$, $R_{c_i}$ is defined in the same way as above,
$R_{d_{ij}}((u_0,\ldots,u_{n-1}))$ iff $u_i=u_j$,
and $R_{s_\sigma}((u_0,\ldots,u_{n-1}),(v_0,\ldots,\allowbreak v_{n-1}))$
iff $u_i=v_{\sigma(i)}$ for each $i<n$.
Let $\sf K_{PA},\sf K_{CA}$ be the class of reducts of structures in ${\sf K}_{PEA}$
to the signatures $(L_{\pa n})_+$ and $(L_{\ca n})_+$, respectively.

We now assume familiarity with the notation of~\cite{Gol95:canon}.
By Theorem 4.5 of \cite{Gol95:canon},
if $\sf K$ is a class of atom structures with
$\mathbb{P}{\sf uK}\subseteq\mathbb{HSU}\sf dK$,
then $\sf SCm\mathbb{SU} dK$ is a canonical variety.
By Theorem 2.2(2,5) of \cite{Gol95:canon},
$\sf PCm=Cm\mathbb{U}d$ and $\mathbb{SU}\sf d=\mathbb{U{\sf d}S}$,
so $\sf SCm\mathbb{SU} dK= SPCm\mathbb{S}K$.
Now let $\sf K\in\{{\sf K_{Df}},{\sf K_{PEA}},{\sf K_{PA},\sf K_{CA}}\}$.
Then $\sf K$ is closed under ultraproducts, and under inner substructures
(since no structure of the above forms has any proper inner substructures),
so $\mathbb{P}{\sf uK}\subseteq\sf K=\mathbb{S}K$.
Consequently, $\sf SPCmK$ --- the closure of
$\{\S^+\mid\S\in\sf K\}$ under subalgebras of products --- is a canonical variety.
But it follows from the definitions that 
${\sf SPCmK_{PEA}}=\rpea n$,
and similar results hold for the other three classes.
\end{proof}

Notwithstanding this proposition, we will show that any first-order
axiomatisation of any of these four varieties requires infinitely many non-canonical sentences.

\section{Algebras from graphs}
\label{S:alggraphs}

Here we will describe how to obtain polyadic equality type algebras from graphs.
In this paper, graphs are undirected and loop-free.
Recall that a set of nodes of a graph is \emph{independent} if 
there is no edge between any two nodes in the set.

\subsection{Atom structures from graphs}
The first
step is given by the following definitions (adapted from \cite[Definition 3.5]{Hirsch09}),
which construct a
polyadic equality atom structure from a graph.

\begin{notation}
We let $Eq(n)$ denote the set of equivalence relations on $n.$
If ${\sim}\in Eq(n)$ and $i < n$, we will write ${\sim_i}$
for the restriction of $\sim$ to $n \setminus \{i\}$.
\end{notation}

\begin{definition}
\index{atom structure from a graph}
Let $\Gamma=(V,E)$ be a graph.
We let 
$\Gamma\times n$ denote the graph 
\[
(V \times n,\{((x,i),(y,j))\in V \times n\mid E(x,y)\vee i\neq j\})
\]
consisting of $n$ copies of $\Gamma$ with all possible additional edges between copies.
\end{definition}

\begin{definition}
Fix a graph $\Gamma$.
Let $S(\Gamma)$ be the
set of all pairs $(K,{\sim})$, where $K : n \to \Gamma \times n$ is
  a partial map and ${\sim}$ an equivalence relation on $n$ that satisfies the
  following:
  \begin{enumerate}
    \item If $|n/{\sim}| = n$, then $\dom(K) = n$ and $\im(K)$ is not independent.
    \item If $|n/{\sim}| = n - 1$, so that there is a unique $\sim$-class $\{i,j\}$ of size 2
      with $i < j < n$, say, then $\dom(K) = \{i,j\}$ and $K(i) = K(j)$.
    \item Otherwise, i.e. if $|n/{\sim}| < n - 1$, $K$ is nowhere defined.
  \end{enumerate}
  \end{definition}
For $(K,\osim),(K',\osim')\in S(\Gamma)$ and $i,j<n$, we will write
$K(i)=K'(j)$ if  either $K(i)$ and $K'(j)$ are both undefined, or
they are both defined and are equal.
According to this, if
$i\sim j$ then $K(i)=K(j)$.

It may be helpful to think of $(K,\osim)\in S(\Gamma)$ as
`really' being $(\widehat{K},\osim)$,
where $\widehat{K}:[n/\osim]^{n-1}\to\Gamma\times n$ is a total map. For notational convenience, we write
$K(j)$ for  $\widehat{K}(\{i_1/\osim,\allowbreak\ldots,\allowbreak i_{n-1}/\osim\})$,
where $\{i_1,\ldots,i_{n-1},j\}=n$ and the right-hand side is defined.

\begin{definition}
Let $i < n$. 
A relation ${\sim}\in Eq(n)$ is said to be \emph{\ds i}
 if $\neg(j \sim k)$ for all distinct
$j,k \in n \setminus \{i\}$.
A pair $(K,\osim)\in S(\Gamma)$ is said to be \emph{\ds i} if $\osim$ is \ds i.
\end{definition}
\begin{remark}
If $(K, {\sim}) \in S(\Gamma)$, then $K$ is defined on $i < n$ if and
only if ${\sim}$ is $i$-distinguishing. 
\end{remark}

 \begin{definition} 
Let $\Gamma$ be a graph.
The polyadic equality atom structure
\[
\At(\Gamma) = (S(\Gamma), D_{ij}, \equiv_i,-^\sigma\mid i,j < n,\;\sigma:n\to n)
\]
is defined as follows:
\begin{enumerate}  
  \item $D_{ij} = \{(K, {\sim}) \in S(\Gamma) \mid i \sim j\}\subseteq S(\Gamma)$, for $i,j < n$.
  
\item $\equiv_i$ is the equivalence relation on $S(\Gamma)$
given by: $(K, {\sim}) \equiv_i (K', {\sim'})$ if and only if $K(i) = K'(i)$ and
    ${\sim_i} = {\sim_i'}$ for $i < n$.
    
    \item For each $\sigma:n\to n$, the map $-^\sigma:S(\Gamma)\to S(\Gamma)$ is given by: $(K,\sim)^\sigma=(K^\sigma,\sim^\sigma)$,
    where \begin{itemize}
\item $\osim^\sigma\in Eq(n)$ is defined by $i\sim^\sigma j$ iff $\sigma(i)\sim \sigma(j)$
    (for $i,j<n$),
    
    \item 
    $K^\sigma(i)$ (for $i<n$) is defined iff $\sim^\sigma$ is \ds i,
    and in that case, $K^\sigma(i)=K(j)$, where $j<n$ is the
     unique element satisfying $j\notin \sigma[n\setminus\{i\}]$.

\end{itemize}
We leave it to the reader to check that $(K^\sigma,\osim^\sigma)$ is well
defined and in $S(\Gamma)$, and that $K^\sigma$
is determined by $K$ and $\sigma$
even though we cannot in general recover $\sim^\sigma$ from them.
Note that if $\sigma$ is one-one then $K^\sigma=K\circ\sigma$.
\end{enumerate}
\end{definition}

\begin{definition}
We write $\A(\Gamma)$ for the 
$n$-dimensional polyadic equality type algebra  $At(\Gamma)^+$.
Explicitly,
\[
\A(\Gamma)=(\wp(S(\Gamma)),\cup,\setminus,\emptyset,S(\Gamma),d_{ij},c_i,s_\sigma
\mid i,j<n,\,\sigma:n\to n),
\]
where $d_{ij}=D_{ij}$  as above, and for $X\subseteq S(\Gamma)$,
\begin{enumerate}
\item $c_iX=\{(K,\osim)\in S(\Gamma)\mid \exists(K',\osim')\in X((K',\osim')\equiv_i(K,\osim))\}$,

\item $s_\sigma X=\{(K,\osim)\in S(\Gamma)\mid(K,\osim)^\sigma\in X\}$.
\end{enumerate}

\end{definition}
$\A(\Gamma)$ is the expansion to the signature of polyadic equality algebras
of a cylindric-type algebra, also written $\A(\Gamma)$, that was defined in \cite{Hirsch09}.
So some results proved for it also apply to the $\A(\Gamma)$ defined above.
Here is one (another is in Proposition~\ref{P:discrim} below):

 \begin{proposition}\label{T:algisca}
Let $\Gamma$ be a graph.
Then the cylindric reduct of $\A(\Gamma)$ is an $n$-dimen\-sional cylindric algebra.
\end{proposition}
\begin{proof}
This is proved in \cite[Claim 3.4 and displayed line (4)]{Kurucz10}.
\end{proof}

\section{Algebra-graph systems}\label{sec:ags}
Proposition~\ref{T:algisca} establishes a relation between graphs and cylindric algebras. However,
we need to study this relationship in a more abstract setting.

\subsection{Definitions}
\begin{definition}
\index{$L_{AGS}$}
We denote by $L_{AGS}$ the signature with three sorts $(\A, \G, \B)$ and the
following symbols:
\begin{enumerate}
\item $\A$-sorted copies of the function symbols $0, 1, +, -, d_{ij}, c_i,s_\sigma$ 
of $L_{\pea n}$ for each $i,j < n$ and $\sigma:n\to n$ (with the obvious
  arities that make $\A$ into a polyadic equality-type algebra);
  
\item $\B$-sorted copies of the function symbols $0,1, +, -$ of $L_{BA}$;

\item a binary (graph edge) relation symbol $E$ on $\G$;
\item a binary relation symbol $H$ on $\G$;
\item a binary relation symbol $\in$ between the elements of $\G$ and $\B$;
\item a unary function symbol $R_i : \A \to \B$ for each $i < n$;
\item a unary function symbol $S_i : \B \to \A$ for each $i < n$.
\end{enumerate}
\end{definition}

We need to pick out certain elements, so that all the elements beneath
are $i$-distin\-guish\-ing and thus have $K(i)$ defined on them.

\begin{definition} \label{D:Fi}
\index{$F_i$}
Let $\A$ be a cylindric-type or polyadic equality-type algebra. For $i < n$, define
\[
F_i = \prod \{-d_{jk}\mid j<k<n, \; j,k \neq i\}.
\]
We generally take $F_i$ to be an element of the algebra under consideration (here, $\A$), though sometimes we regard
it as an $L_{AGS}$-term.
\end{definition}
\begin{remark}
Clearly, for an algebra from a graph $\A(\Gamma)$, $F_i$ is just the
sum of all the $i$-distinguishing atoms.
For $(K,\osim)\in S(\Gamma)$, $K(i)$ is defined iff $(K,\osim)\in F_i$.
\end{remark}

\begin{definition} \label{D:MGamma}
For a graph $\Gamma$, let $M(\Gamma)$ be the $3$-sorted $L_{AGS}$ structure
\[
(\A(\Gamma),\; \Gamma \times n,\; \wp(\Gamma \times n))
\]
with operations defined as follows:
\begin{itemize}
\item The $A$-sorted and $\B$-sorted symbols are interpreted on $\A(\Gamma)$, $\wp(\Gamma\times n)$
in the natural way.

\item $E$ is interpreted as the edge relation on $\Gamma \times n$.

\item We have $H(x,y)$ if and only if there is $\ell < n$ such that
$x,y \in \Gamma \times \{\ell\}$. 

\item The relation $\in$ denotes membership of
elements of $\Gamma \times n$ in the sets that are elements of $\wp(\Gamma \times n)$.

\item Finally, we have
\begin{align*}
R_i(a)&=\{K(i) \mid (K,{\sim}) \in F_i \cdot a\} &\mbox{for }a\in\A(\Gamma),\\
S_i(B)&= \{(K, {\sim}) \in F_i \mid K(i) \in B\}&\mbox{for }B\in\wp(\Gamma\times n).
\end{align*}
\end{itemize}
\end{definition}

We now define a theory that helps us talk about the subclass of all the
$L_{AGS}$-structures similar to the ones derived from graphs.

\begin{definition} \label{D:ags}
\index{algebra-graph system}
A (first-order) $L_{AGS}$-formula is said to be  \emph{$\A$-universal}
if it is of the form
\[
(\forall\vec xm:\cal A)\;\varphi,
\]
where $\varphi$ is an $L_{AGS}$-formula with
 no quantifiers over the $\c A$-sort.
 
We define $\U$ to be the set of $\A$-universal sentences that are true in all
$L_{AGS}$-structures $M(\Gamma)$ for graphs $\Gamma$.
An  $L_{AGS}$-structure $M$ that is a model of $\U$ is called an
\emph{algebra-graph system}.
\end{definition}
This definition ensures that a good number of first-order statements that hold for
algebras from graphs, also hold in any algebra-graph system. It will allow us
to prove many first-order statements for algebra-graph systems, by just showing they
are $\A$-universal and
hold in  $M(\Gamma)$ for every graph $\Gamma$. We will refer to this approach 
as the \emph{generalisation technique}.

\subsection{Basic properties of algebra-graph systems}

\begin{lemma}\label{lem:A CA}
In any algebra-graph system $M=(\A, \G, \B)$, the cylindric reduct of
$\A$ is a cylindric algebra, and $\B$ is a boolean algebra
isomorphic to a subalgebra of $\wp(\G)$. 
\end{lemma} 

\begin{proof}
The first statement follows by the generalisation technique, as we know
from Proposition~\ref{T:algisca} that an arbitrary algebra from a graph will
satisfy all the axioms for cylindric algebras.
These axioms are equations and can be recast in the obvious way 
as $\A$-universal $L_{AGS}$-sentences.
A similar argument shows that $\B$ is a boolean algebra.
Since the $\A$-universal sentences 
\begin{align*}
\forall B,B':\B&(\forall p:\G(p\in B\leftrightarrow p\in B')\to B=B'),
\\
\forall p:\G&(p\in 1\wedge\neg(p\in 0)),
\\
\forall B,B':\B\;\forall p:\G&(p\in B+B'\leftrightarrow p\in B\vee p\in B'),
\\
\forall B:\B\;\forall p:\G&(p\in -B\leftrightarrow\neg(p\in B))
\end{align*}
are in $\U$ and so are true in $M$,
the function $B\mapsto\{p\in\G\mid M\models p\in B\}$ is a boolean embedding
from $\B$ into $\wp(\G)$.
\end{proof}
So in any algebra-graph system $(\A,\G,\B)$,
Lemma~\ref{lem:A CA} allows us to regard $\B$ as a boolean algebra of subsets of $\G$,
and the $L_{AGS}$-relation symbol `$\in$' as denoting genuine set membership.

Recall that $F_i= \prod_{{j<k<n},\,{ j,k \neq i}} -d_{jk}$ from Definition~\ref{D:Fi}. 

\begin{lemma} \label{L:FiFj}
Let $M = (\A,\G,\B)$ be an algebra-graph system and $i,j < n$.
Then $F_i \cdot d_{ij} \leq F_j$ holds in $\A$.
\end{lemma}
\begin{proof}
It is enough to prove the lemma for algebras from graphs, as it is
clearly a set of $\A$-universal first-order sentences. 
But this is easy and was done in \cite[Lemma 4.2]{Hirsch09}.
It is also  easily seen to hold in cylindric algebras,
of which $\A$ is one (by Lemma~\ref{lem:A CA}).
\end{proof}

Now we examine the functions $R_i,S_i$.

\begin{lemma} \label{L:RSProps}
Let $M = (\A,\G,\B)$ be an algebra-graph system and let $i,j < n$ be distinct. Then:
\begin{enumerate}
 \renewcommand{\theenumi}{(\roman{enumi})}
\renewcommand{\labelenumi}{(\roman{enumi})}

\item \label{RSProps-1}  If $a,b\in\A$ and $a \leq b$ then $R_i(a) \leq R_i(b)$.

\item\label{RSProps0}   If $b \in \A$  and $b \leq d_{ij}$ then $R_i(b) = R_j(b)$.

\item\label{RSProps1}  If $a \in \A$ and $a \leq F_i$, then $S_i(R_i(a)) \geq a$.

\item\label{RSProps2} The map $f:\A\to\B$ given by
$f(a)= R_i(a\cdot d_{ij})$ is a
boolean homomorphism satisfying $f(F_i\cdot d_{ij})=1$.

\item\label{RSProps3}  If $B \in \B$, then $R_i(S_i(B)) = B$.
(Hence, $S_i$ is injective and $R_i$  surjective.)

\item\label{RSProps4} If $B\in\B$ then $c_iS_i(B)=S_i(B)$.
\end{enumerate}
\end{lemma}
\begin{proof}
It is again sufficient to show that this is true for any structure $M(\Gamma)$ from
Definition~\ref{D:MGamma}.
Parts~\ref{RSProps-1} and~\ref{RSProps0} are easy and left to the reader.

\ref{RSProps1} Let $(K,{\sim}) \in a\leq F_i$ be arbitrary. Recall that
\[
R_i(a) = \{K(i) \mid (K, {\sim}) \in a \cdot F_i\}.
\]
Then $(K, {\sim}) \in F_i$, so $K(i)$ is defined,  $K(i) \in R_i(a)$, and hence
\[
(K, {\sim}) \in \{(K', {\sim'}) \in F_i \mid K'(i) \in R_i(a)\} = S_i(R_i(a)).
\]
This shows that $a \leq S_i(R_i(a))$.

\medskip

\ref{RSProps2} First, observe that \begin{itemize}
\item [$(\ddag)$]
for any $p\in\Gamma\times n$, there is a unique atom
$(K_p,\approx)\in F_i\cdot d_{ij}$
with $K_p(i)=p$.
\end{itemize}
  For, we may define ${\approx}\in Eq(n)$ to be the
(unique) $i$-distinguishing relation with $i \approx j$ and define $K_p$ by
\[
K_p(i) = K_p(j) = p, \qquad K_p(k) \text{ undefined if } k \neq i,j.
\]
Then $(K_p, {\approx})$ is certainly a valid element of $\At(\Gamma)$ contained in $F_i$ and $d_{ij}$ and with $K_p(i)=p$, and it is clearly the only such atom.

Returning to the lemma, it is clear that $f(0)=0$ and $f(a+b)=f(a)+f(b)$ for all $a,b\in\A$.
Let $p\in\Gamma\times n$ be arbitrary.
For any $a\in\A$, we have $p\in f(a)$ iff $(K_p,\approx)\in a$.
Hence, $p\notin f(a)$ iff $(K_p,\approx)\notin a$,
iff $(K_p,\approx)\in -a$, iff $p\in f(-a)$.
This shows that $f(-a)=-f(a)$.  Hence also, $f(1)=1$.
So $f$ is a boolean homomorphism.
If $p\in\Gamma \times n$ then $(K_p,\approx)\in F_i\cdot d_{ij}$
and so $p\in f(F_i\cdot d_{ij})$. As $p$ was arbitrary,
$f(F_i\cdot d_{ij})=1$.

\medskip

\ref{RSProps3} Let $B \in \B$. First note that
\[
R_i(S_i(B)) = \{K(i) \mid (K, {\sim}) \in S_i(B)\}
            = \{K(i) \mid (K, {\sim}) \in F_i, K(i) \in B\} \subseteq B.
\]
For the converse,   let $p \in B$ be given.
By $(\ddag)$ above,  $(K_p,\approx)\in S_i(B)$,
and $p = K_p(i) \in R_i(S_i(B))$.
This shows that $B \subseteq R_i(S_i(B))$.

\medskip

\ref{RSProps4} Let $(K,\osim)\equiv_i(K',\osim')\in S_i(B)$,
so that $(K',\osim')\in F_i$ and $K'(i)\in B$.
Then $\osim_i=\osim'_i$, so $(K,\osim)\in F_i$ as well,
and $K(i)=K'(i)\in B$.
Consequently, $(K,\osim)\in S_i(B)$.
Hence, $c_iS_i(B)\subseteq S_i(B)$, and the converse is trivial.
\end{proof}

Next, we examine the substitution operators.

\begin{definition}
For $\osim\in Eq(n)$
let $d_\osim=\prod_{i,j<n,i\sim j}d_{ij}\cdot\prod_{i,j<n,i\not\sim j}-d_{ij}$.
\end{definition}

\begin{lemma}\label{lem:sub tech}
Let $M = (\A,\G,\B)$ be an algebra-graph system,
 $a\in\A$, $i,j<n$,  and $\sigma,\tau:n\to n$.
\begin{enumerate}
 \renewcommand{\theenumi}{(\roman{enumi})}
\renewcommand{\labelenumi}{(\roman{enumi})}

\item\label{lem:sub tech1} The map $s_\sigma:\A\to\A$ is a boolean homomorphism.

\item\label{lem:sub tech1.5} $s_{\sigma\circ\tau}=s_\sigma\circ s_\tau$.

\item\label{lem:sub tech2} $s_\sigma d_{ij}=d_{\sigma(i)\sigma(j)}$.

\item\label{lem:sub tech3}  If $\osim\in Eq(n)$, then $s_\sigma d_{\osim^\sigma}\geq d_\osim$.

\item\label{lem:sub tech4}
If $\sigma[n\setminus\{i\}]=n\setminus\{j\}$,
then $R_j(s_\sigma a)\leq R_i(a)$.

\item\label{lem:sub tech5} If $i\notin \im\sigma$ then $c_is_\sigma a=s_\sigma a$,
and if $\sigma$ is one-one then $c_{\sigma(i)}s_\sigma a=s_\sigma c_ia$.

\end{enumerate}
\end{lemma}

\begin{proof}
Again, it is enough to show that the lemma is true for
an arbitrary algebra-graph system $M(\Gamma)$ from a graph $\Gamma$, as all statements
are
definable by $\A$-universal first-order sentences.

\ref{lem:sub tech1}  By the definitions,
$s_\sigma\emptyset=\emptyset$,
$s_\sigma1=1$, and for any $a,b\in\A(\Gamma)$,
$s_\sigma(a+b)=\{\kappa\in S(\Gamma)\mid\kappa^\sigma\in a+b\}=
\{\kappa\mid\kappa^\sigma\in a\}\cup\{\kappa\mid\kappa^\sigma\in b\}=s_\sigma a+s_\sigma b$,
and $s_\sigma(-a)=\{\kappa\in S(\Gamma)\mid \kappa^\sigma\in-a\}
=S(\Gamma)\setminus\{\kappa\mid\kappa^\sigma\in a\}=-s_\sigma a$.

\medskip

\ref{lem:sub tech1.5} Let $(K,\osim)\in S(\Gamma)$ be arbitrary.
We claim that $((K,\osim)^\sigma)^\tau=(K,\osim)^{\sigma\circ\tau}$:
that is, 
\[
((K^\sigma)^\tau,(\osim^\sigma)^\tau)=(K^{\sigma\circ\tau},\osim^{\sigma\circ\tau}).
\]
For $i,j<n$,
plainly $i\mathrel{(\osim^\sigma)^\tau}j$
iff $\tau(i)\sim^\sigma\tau(j)$ iff $\sigma(\tau(i))\sim\sigma(\tau(j))$
iff $i\sim^{\sigma\circ\tau}j$,
so $(\osim^\sigma)^\tau=\osim^{\sigma\circ\tau}$.
Let $i<n$.
Then $(K^\sigma)^\tau(i)$ is defined
iff $(\osim^\sigma)^\tau$ is \ds i,
iff $\osim^{\sigma\circ\tau}$ is \ds i,
iff $K^{\sigma\circ\tau}(i)$ is defined.
In that case, 
$(K^\sigma)^\tau(i)=K^\sigma(j)$ where $j\notin\tau[n\setminus\{i\}]$.
Then $\osim^\sigma$ is \ds j
and $K^\sigma(j)=K(k)$ where $k\notin\sigma[n\setminus\{j\}]$.
But now, $k\notin\sigma\circ\tau[n\setminus\{i\}]$,
so $K^{\sigma\circ\tau}(i)=K(k)$ as well.
This proves the claim.
Consequently,
$s_\sigma s_\tau a=\{\kappa\in S(\Gamma)\mid \kappa^\sigma\in s_\tau a\}
=\{\kappa\mid (\kappa^\sigma)^\tau\in a\}=\{\kappa\mid \kappa^{\sigma\circ\tau}\in a\}
=s_{\sigma\circ\tau}a$.

\ref{lem:sub tech2} We have
$s_\sigma d_{ij}=\{(K,\osim)\mid(K^\sigma,\osim^\sigma)\in d_{ij}\}
=\{(K,\osim)\mid\sigma(i)\sim\sigma(j)\}=d_{\sigma(i)\sigma(j)}$.

\ref{lem:sub tech3}
By \ref{lem:sub tech1} and \ref{lem:sub tech2},
\begin{align*}
s_\sigma d_{\osim^\sigma} &=
s_\sigma\Big(\prod_{i,j<n,\,i\sim^\sigma j}d_{ij}\cdot\prod_{i,j<n,\,i\not\sim^\sigma j}-d_{ij}\Big)
=\prod_{i,j<n,\,i\sim^\sigma j}\hskip-6pt
s_\sigma d_{ij}\cdot\prod_{i,j<n,i\not\sim^\sigma j}\hskip-6pt-s_\sigma d_{ij}
\\
&=\prod_{i,j<n,\,\sigma(i)\sim\sigma( j)}d_{\sigma(i)\sigma(j)}\cdot
\prod_{i,j<n,\,\sigma(i)\not\sim\sigma( j)}-d_{\sigma(i)\sigma(j)}.
\end{align*}
The last expression comprises some of the conjuncts (all of them, if $\sigma$ is onto)
of $d_\osim$.
So $s_\sigma d_{\osim^\sigma}\geq d_\osim$ as required
(with equality if $\sigma$ is onto).

\ref{lem:sub tech4}
Let  $p\in R_j(s_\sigma a)$,
so that $p=K(j)$ for some $(K,\osim)\in s_\sigma a\cdot F_j$.
Hence, $(K^\sigma,\osim^\sigma)\in a$, and  $\osim$ is \ds j.
As $\sigma[n\setminus\{i\}]=n\setminus\{j\}$,
it follows that $\osim^\sigma$ is \ds i.
So $(K^\sigma,\osim^\sigma)\in a\cdot F_i$, and  $K^\sigma(i)$ is defined and is plainly $K(j)$, ie.\ $p$.
Hence $p\in R_i(a)$ as required.

\ref{lem:sub tech5} Let $i\notin\im\sigma$ and $(K,\osim)\equiv_i(K',\osim')\in s_\sigma a$,
so that $(K'^\sigma,\osim'^\sigma)\in a$.
We show that $(K^\sigma,\osim^\sigma)=(K'^\sigma,\osim'^\sigma)$.
As $\osim_i=\osim'_i$ and $i\notin\im\sigma$, we have $\osim^\sigma=\osim'^\sigma$.
Take $j<n$ such that $\osim^\sigma$ is \ds j.
Then plainly $i\notin\sigma[n\setminus\{j\}]$,
so $K^\sigma(j)=K(i)=K'(i)=K'^\sigma(j)$.
It follows that $(K^\sigma,\osim^\sigma)=(K'^\sigma,\osim'^\sigma)\in a$,
so $(K,\osim)\in s_\sigma a$.
This proves that $c_is_\sigma a\leq s_\sigma a$. The converse is immediate by 
Lemma~\ref{lem:A CA}.

Now suppose that $\sigma:n\to n$ is one-one.
Then plainly, for any atoms $(K,\osim)$, $(K,\osim')$, we have
$(K^\sigma,\osim^\sigma)=(K\circ\sigma,\osim^\sigma)$,
and $(K,\osim)\equiv_{\sigma(i)}(K',\osim')$
iff $(K\circ\sigma,\osim^\sigma)\equiv_i(K'\circ\sigma,\osim'^\sigma)$.

Let $(K,\osim)$ be arbitrary.
Then $(K,\osim)\in c_{\sigma(i)}s_\sigma a$
iff there is $(K',\osim')$ with $(K,\osim)\allowbreak\equiv_{\sigma(i)}(K',\osim')$
and $(K'\circ\sigma,\osim'^\sigma)\in a$,
iff there is $(K',\osim')$ with $(K\circ\sigma,\osim^\sigma)\equiv_{i}(K'\circ\sigma,\osim'^\sigma)$
and $(K'\circ\sigma,\osim'^\sigma)\in a$,
iff there is $(K^*,\osim^*)$ with $(K\circ\sigma,\osim^\sigma)\equiv_{i}(K^*,\osim^*)\in a$,
iff $(K\circ\sigma,\osim^\sigma)\in c_ia$, iff
$(K,\osim)\in s_\sigma c_ia$ as required.
 \end{proof}

\subsection{Simple algebras}
Recall  that a cylindric algebra $\A$ is
\emph{simple} if $|\A| > 1$ and for any algebra $\A'$ with cylindric signature,
any homomorphism $\varphi : \A \to \A'$ is either trivial or injective.
We will see that the cylindric reduct of the algebra part of an algebra-graph
system is simple, so that if it is representable, it
has  a representation that is  just an embedding into
a single cylindric set algebra.

\begin{definition}
\index{discriminator term}
Let $\C$ be a class of BAOs of the same signature $L$. An $L$-term $d(x)$
satisfying
\[
d(a) = \begin{cases}
1 & \text{if } a > 0, \\
0 & \text{if } a = 0.
\end{cases}
\]
for each $a \in \A \in \C$, is called a \emph{discriminator term for $\C$}.
\end{definition}

\begin{proposition} \label{P:discrim}
The class $\{\A(\Gamma) \mid \Gamma \text{ a graph}\}$ has a discriminator term,
namely, $c_1 \dots c_{n-1}\allowbreak c_{n-1} \dots c_{1}x$.
\end{proposition}
\begin{proof}
See line (5) in the proof of \cite[Lemma 5.1]{Hirsch09}.
\end{proof}
We deduce the following in a standard way.
\begin{corollary} \label{C:simple}
In every algebra-graph system $ (\A, \G, \B)$, the cylindric-type reduct 
$\A\restr{L_{\ca n}}$  of $\A$ is simple, as is each of its subalgebras.
\end{corollary}

\begin{proof}
Let $\A'$ an algebra with cylindric signature, $\A^*\subseteq \A\restr{L_{\ca n}}$, and
$\varphi : \A^* \to \A'$ a homomorphism.
It follows from Proposition~\ref{P:discrim}, by the generalisation technique,
that $\A^*$ has a discriminator term $d(x)=c_1 \dots c_{n-1} c_{n-1} \dots\allowbreak c_{1}x$.
Suppose $\varphi$ is not injective, i.e. there are distinct $a,b \in \A^*$ such
that $\varphi a = \varphi b$. Then $(a - b) + (b - a) \neq 0$ and therefore
\begin{align*}
1 = \varphi d((a - b) + (b - a)) 
  &= d((\varphi a - \varphi b) + (\varphi b - \varphi a)) 
  = d(0) 
  = 0.
\end{align*}
Thus $\varphi$ is trivial if it is not injective.
\end{proof}

\begin{lemma} \label{L:easyembed}
Let $\A \in \rca{n}$ be a representable cylindric algebra. If $\A$ is simple,
then it has a representation that is an embedding into a single cylindric set algebra.
\end{lemma}
\begin{proof}
There is a representation $h : \A \to \prod_{k \in K} \S_k$, where $K$ is an
index set and for each $k \in K$, $S_k$ is a non-empty base set and
\[
\S_k = (\wp(^nS_k), \cup, \setminus, \emptyset, ^nS_k, D_{ij}^k, C_i^k\mid i,j < n)
\]
Because $h$ is injective and $|\A| > 1$, the index set $K \neq \emptyset$. So
choose $\ell \in K$ and let $\pi$ be the projection
of $\prod_{k \in K} \S_k$ onto $S_\ell$. Then $\pi \circ h$ is certainly a
homomorphism and because
\[
\pi \circ h (1) = {}^nS_\ell \neq \emptyset = \pi \circ h (0),
\]
it is non-trivial. But because $\A$ is simple, $\pi \circ h:\A\to\S_\ell$ is injective and
thus a representation that is an embedding into a single 
cylindric set algebra.
\end{proof}

\section{Ultrafilters}
\label{S:netpatchsys}

We now examine ultrafilters in algebra-graph systems.

\subsection{Ultrafilter structures from algebra-graph systems}

By the generalisation technique, if $M=(\A,\G,\B)$ is an algebra-graph system
then $\A$ is an $L_{\pea n}$-BAO, so its ultrafilter
structure $\A_+$ (see Definition~\ref{def:uf str}) is defined; it satisfies
\[
\begin{array}{rcl}
R_{d_{ij}}(\nu)&\iff & d_{ij}\in\nu,
\\
R_{c_i}(\mu,\nu)&\iff& c_i[\mu] = \{c_i a \mid a \in \mu\} \subseteq \nu,
\\
R_{s_\sigma}(\mu,\nu)
&\iff& s_\sigma[\mu]= \{s_\sigma a \mid a \in \mu\} \subseteq\nu.
\end{array}
\]
We view $\A_+$ as a polyadic equality atom structure (Definition~\ref{def:polyas}) by defining
\[
\begin{array}{rcl}
\mu \equiv_i \nu &\iff &R_{c_i}(\mu,\nu),
\\
\nu^\sigma&=&\{a\in\A\mid s_\sigma a\in\nu\}.
\end{array}
\]
Lemma~\ref{lem:ss uf}, the comment following it, and Lemma~\ref{L:ultraproj}\ref{ultraproj5} show that this 
gives a well defined polyadic equality atom structure which, when
regarded as an $L_{\pea n}$-atom structure as in Definition~\ref{def:polyas},
yields the ultrafilter structure $\A_+$ as above.

\begin{lemma}\label{lem:ss uf}
Let $M = (\A,\G,\B)$ be an algebra-graph system  and $\sigma,\tau:n\to n$.
Then for any ultrafilter $\nu$ of $\A$,
the set $\nu^\sigma$ is also an ultrafilter of $\A$, 
and $\nu^{\sigma\circ\tau}=(\nu^\sigma)^\tau$.
\end{lemma}

\begin{proof}
By Lemma~\ref{lem:sub tech}, $s_\sigma:\A\to\A$ is a boolean homomorphism.
It is well known and easily seen that 
for boolean algebras $\B_1,\B_2$, the preimage of an ultrafilter 
of $\B_2$ under a boolean homomorphism $f:\B_1\to \B_2$ is an ultrafilter of $\B_1$.
So $\nu^\sigma$ is an ultrafilter of $\A$.
By Lemma~\ref{lem:sub tech}\ref{lem:sub tech1.5},
\begin{align*}
\nu^{\sigma\circ\tau}
=\{a\in\A\mid s_{\sigma\circ\tau}a\in\nu\}
&=\{a\in\A\mid s_\sigma(s_\tau a)\in\nu\}\\
&=\{a\in\A\mid s_\tau a\in \nu^\sigma\}
=(\nu^\sigma)^\tau.
\end{align*}
\end{proof}
It follows that $R_{s_\sigma}(\mu,\nu)$ iff 
$\mu\subseteq\nu^\sigma$, iff $\nu^\sigma=\mu$ since both are ultrafilters.

\subsection{Projections of ultrafilters}

\begin{definition}
Let $M = (\A,\G,\B)$ be an algebra-graph system, 
 let  $\mu$ be an ultrafilter of  $\A$, and let $i < n$.
We write $\mu(i)$ for the set
$R_i[\mu]=\{R_i(a) \mid a \in \mu\} \subseteq \B$
--- the `$i$th projection of $\mu$'.
We say that $\mu$ is \emph{$i$-distinguishing} if it 
contains $F_i$. 

\end{definition} 
Clearly, $\mu$ is  \ds i iff it does not contain any of the $d_{jk}$ for
distinct $j,k \in n \setminus \{i\}$.  In this case, $\mu(i)$ turns out to be an ultrafilter
of $\c B$.  The following lemma establishes this and other 
facts about projections of ultrafilters.

\begin{lemma} \label{L:ultraproj}
Let $M = (\A,\G,\B)$ be an algebra-graph system, let $i<n$, and let $\mu,\nu$
be ultrafilters of $\A$. 
\begin{enumerate}
\renewcommand{\theenumi}{(\roman{enumi})}
\renewcommand{\labelenumi}{(\roman{enumi})}

\item\label{ultraproj1} The projection $\mu(i)$ is an ultrafilter of $\B$ if $\mu$ is $i$-distinguishing, and $\B$ (that is, the improper filter on $\B$), otherwise.

\item\label{ultraproj2} If $j<n$ and $d_{ij} \in \mu$, then $\mu(i) = \mu(j)$.

\item\label{ultraproj3} If $i\neq j<n$ and $\beta$ is an ultrafilter of $\B$, 
then  $\alpha=\{a\in\A\mid R_i(a\cdot d_{ij})\in \beta\}$ is the unique
ultrafilter of $\A$ with $F_i,d_{ij}\in\alpha$ and $\alpha(i)=\beta$.

\item\label{ultraproj4}  $\mu \equiv_i \nu$ iff
(a) $d_{jk}\in\mu$ iff $d_{jk}\in\nu$ for all $j,k\in n\setminus\{i\}$,  and (b)  $\mu(i) = \nu(i)$.

\item\label{ultraproj5}  $\equiv_i$ is an equivalence relation on $\A_+$.

\item\label{ultraproj6}
If $\sigma:n\to n$ and $\sigma[n\setminus\{i\}]=n\setminus\{j\}$,
then $\mu^\sigma(i)=\mu(j)$.
\end{enumerate}
\end{lemma}

\begin{proof}
\ref{ultraproj1} 
If $F_i\in\mu$, then $\mu(i)=\{B\in\B\mid S_i(B)\in\mu\}$.
For, if $S_i(B)\in\mu$ then by Lemma~\ref{L:RSProps}\ref{RSProps3}, $B=R_i(S_i(B))\in\mu(i)$.
Conversely, if $B\in\mu(i)$ then $B=R_i(a)$ for some $a\in\mu$ with $a\leq F_i$
(since $F_i\in\mu$).
By Lemma~\ref{L:RSProps}\ref{RSProps1}, $S_i(B)=S_i(R_i(a))\geq a$ so $S_i(B)\in\mu$.
Let $\A_i$ be the relativisation of the boolean reduct of
$\A$ to $F_i$. It is easily seen by the generalisation technique
that $S_i:\B\to \A_i$ is a boolean homomorphism.
Now $F_i\in\mu$, so $\mu\cap \A_i$ is an ultrafilter of $\A_i$.
So its preimage under $S_i$, namely $\mu(i)$, is an ultrafilter of~$\B$.

If  $-F_i\in\mu$, then for any $B\in\B$
we have $-F_i+S_i(B)\in\mu$.
By the generalisation technique, $ R_i(-F_i+a)=R_i(a)$ for all $a$, so by Lemma~\ref{L:RSProps}\ref{RSProps3},
 $R_i(-F_i+S_i(B))=R_i(S_i(B))=B$. So $B\in\mu(i)$.
 Since $B$ was arbitrary, $\mu(i)=\B$.

\medskip

\ref{ultraproj2} This is obvious if $i = j$, so suppose $i \neq j$. Assume $d_{ij} \in \mu$.
Let $R_i(a)$ be an element of
$\mu(i)$ for some $a \in \mu$. Define
$b = a\cdot d_{ij} \in \mu$. 
It follows from Lemma~\ref{L:RSProps}\ref{RSProps-1},\ref{RSProps0} 
that %
$R_i(a) \geq R_i(b) = R_j(b) \in \mu(j)$.
By \ref{ultraproj1}, $\mu(j)$ is always a filter, so $R_i(a)\in\mu(j)$.
Thus $\mu(i) \subseteq \mu(j)$.  The converse inclusion holds by symmetry, so
$\mu(i) = \mu(j)$.

\medskip
\ref{ultraproj3}
By Lemma~\ref{L:RSProps}\ref{RSProps2},
the map $a\mapsto R_i(a\cdot d_{ij})$ is a boolean homomorphism from $\A$ to $\B$.
As $\alpha$ is the preimage of $\beta$ under this map, it
is an ultrafilter of~$\A$.
The lemma also shows that 
 $R_i(F_i\cdot d_{ij})=1$,
so $F_i\cdot d_{ij}\in\alpha$.
Plainly, $\alpha(i)\subseteq\beta$, so 
as $\beta$ is an ultrafilter of $\B$, by \ref{ultraproj1} we have $\alpha(i)=\beta$.

Let $\alpha'$ be any ultrafilter of $\A$ with $F_i,d_{ij}\in\alpha'$ and $\alpha'(i)=\beta$.
If $a\in\alpha'$, then $a\cdot d_{ij}\in\alpha'$,
so $R_i(a\cdot d_{ij})\in\beta$. Hence,  
$a\in\{a\in\A\mid R_i(a\cdot d_{ij})\in\beta\}=\alpha$.
So $\alpha'\subseteq\alpha$, and since both sides are ultrafilters of $\A$, they are equal. 

\medskip

\ref{ultraproj4} (${\implies}$) Assume $\mu \equiv_i \nu$. 
For each $j,k\neq i$, we have
$d_{jk}\in\mu\Rightarrow d_{jk}=c_id_{jk}\in\nu$,
and 
$-d_{jk}\in\mu\Rightarrow -d_{jk}=c_i{-}d_{jk}\in\nu$
(these equations are easily established 
by the generalisation technique or using basic properties of cylindric algebras:
see, e.g., \cite[1.3.3, 1.2.12]{Henkin71}).
As $\mu$ and $\nu$ are ultrafilters, this proves (a).
Hence also, $F_i\in\mu$ iff $F_i\in\nu$.

We prove (b).  
If $-F_i\in\mu$, part~\ref{ultraproj1} gives $\mu(i) =\B= \nu(i)$, proving (b).
Assume then that  $F_i\in\mu$. Then  $\mu(i)$ and $ \nu(i)$
 are ultrafilters by part~\ref{ultraproj1}, so it is enough to show $\mu(i) \subseteq \nu(i)$.
Let $B\in\mu(i)$ be arbitrary.
Take $a\in\mu$  such that
$B=R_i(a)$. By assumption, $c_i a \in \nu$. 
Note that the following holds for
all algebras from graphs:
\[
\forall a : \A (R_i(a) = R_i(c_i a)).
\]
So $B=R_i(a) = R_i(c_i a) \in \nu(i)$ as required. 

(${\impliedby}$) For the converse, assume the hypotheses and
let 
\[
D=\prod_{{j,k\in n\setminus\{ i\}},\;{d_{jk}\in\mu}}d_{jk}
\cdot\prod_{{j,k\in n\setminus\{ i\}},\;{d_{jk}\notin\mu}}-d_{jk},
\]
so $D\in\mu\cap\nu$ by (a).
Now the following holds in algebras from graphs:
\[
\forall a,b:\A(0<a\leq D\wedge R_i(b)\leq R_i(a)\to b\cdot D\leq c_ia).
\]
For let $(K,\osim)\in b\cdot D$.
If $K(i)$ is defined, then $K(i)\in R_i(b)\leq R_i(a)$, so we may pick $(K',\osim')\in a$
with $K'(i)=K(i)$.
If $K(i)$ is undefined, let $(K',\osim')\in a$ be arbitrary (we use $a>0$ here).
Since $(K,\osim),(K',\osim')\in D$, we have $\osim_i=\osim'_i$,
hence in the second case $K'(i)$ is also undefined and $K'(i)=K(i)$.
So $(K,\osim)\equiv_i(K',\osim')$, yielding $(K,\osim)\in c_ia$.

By the generalisation technique, the statement holds for $M$.
So if $a\in\mu$,
then $a\cdot D\in\mu$
and $R_i(a\cdot D)\in\mu(i)=\nu(i)$,
so $R_i(a\cdot D)=R_i(b)$ for some $b\in\nu$.
By the above, $b\cdot D\leq c_i(a\cdot D)\leq c_ia$, and as $b\cdot D\in\nu$,
we have $c_ia\in\nu$ as well.  So $\mu\equiv_i\nu$ by definition.

\medskip

\ref{ultraproj5} Immediate from \ref{ultraproj4}.

\medskip

\ref{ultraproj6} Let $B\in\mu^\sigma(i)$.
Then $B=R_i(a)$ for some $a\in\mu^\sigma$, so
$s_\sigma a\in\mu$ and $R_j(s_\sigma a)\in\mu(j)$.
By Lemma~\ref{lem:sub tech}\ref{lem:sub tech4},
which applies since $\sigma[n\setminus\{i\}]=n\setminus\{j\}$,
we have $R_j(s_\sigma a)\leq R_i(a)=B$.
As $\mu(j)$ is a filter,
 $B\in\mu(j)$ as well.
As $B$ was arbitrary, $\mu^\sigma(i)\subseteq\mu(j)$.

So by part~\ref{ultraproj1}, it only remains to show that
if $\mu^\sigma(i)$ is an ultrafilter of $\B$ then so is $\mu(j)$.
But as $\sigma[n\setminus\{i\}]=n\setminus\{j\}$, the definition
of $F_i$ and Lemma~\ref{lem:sub tech}\ref{lem:sub tech1}\ref{lem:sub tech2} 
yield 
\[s_\sigma F_i
=s_\sigma\left(\prod_{k,l\in n\setminus\{i\},\,k\neq l}-d_{kl}\right)
=\hskip-3pt\prod_{k,l\in n\setminus\{i\},\,k\neq l}\hskip-6pt-d_{\sigma(k)\sigma(l)}
=\hskip-3pt\prod_{k,l\in n\setminus\{j\},\,k\neq l}\hskip-6pt-d_{kl}
=F_j.
\]
So $F_i\in\mu^\sigma$ iff  $s_\sigma F_i\in \mu$, iff $F_j\in\mu$.
By part~\ref{ultraproj1}, 
$\mu^\sigma(i)$ is an ultrafilter of $\B$ iff $\mu(j)$ is.
\end{proof}

\section{Networks and patch systems}\label{sec:nks}

In this section we introduce approximations to representations,  called
ultrafilter networks. They will be part of the game to construct
representations.
We will approximate the networks themselves by lower-dimensional objects that we call
patch systems.

\subsection{Ultrafilter networks}

\begin{definition}
\index{i-distinguishing}
Let $X$ be a set, $i < n$, and $v \in {}^nX$.
\begin{enumerate}
\item For $w \in {}^nX$, we say $v \equiv_i w$ if $v_j = w_j$ for all $j< n$,
$j \neq i$.
\item If $v_j \neq v_k$ for all distinct $j, k \in n \setminus \{i\}$, then
$v$ is called $i$-\emph{distinguishing}.
\end{enumerate}
\end{definition}

\begin{definition}\label{def:uf nk}
\index{partial ultrafilter network}
\index{ultrafilter network}
Let $M = (\A,\G,\B)$ be an algebra-graph system.
A \emph{cylindric ultrafilter network} over $\A$ is a pair $\N = (N_1, N_2)$,
where $N_1$ is a set and $N_2 : {}^nN_1 \to \A_+$ is a  map that satisfies
the following for any $v, w \in  {}^nN_1$:
\begin{enumerate}
\item For $i,j < n$, we have $d_{ij} \in N_2(v)$ if and only if $v_i = v_j$.
\item If $i < n$ and $v \equiv_i w$, then $N_2(v) \equiv_i N_2(w)$.
\end{enumerate}
$\c N$ is said to be a \emph{polyadic ultrafilter network} if in addition:
\begin{enumerate}
\setcounter{enumi}{2}
\item For each $\sigma:n\to n$  we have
$N_2(v\circ\sigma)=N_2(v)^\sigma$.
\end{enumerate}

If $\N = (N_1, N_2)$ and $\M = (M_1, M_2)$ are  ultrafilter networks,
we write $\N \subseteq \M$ to denote $N_1 \subseteq M_1$ and
$M_2\restriction{}^nN_1=N_2$. 
For a chain $\N^0\subseteq\N^1\subseteq\cdots$ of ultrafilter networks $\N^k=(N^k_1,N^k_2)$,
we write $\bigcup_{k<\omega}\N_k$ for the ultrafilter network
$(\bigcup_{k<\omega}N^k_1,\bigcup_{k<\omega}N^k_2)$
(here we view the maps $N^k_2$ formally as sets of ordered pairs).
We will often write $\N$ for both $N_1$ and $N_2$.
\end{definition}

\subsection{Patch systems}

Patch systems provide a way to assign ultrafilters on a graph to
$(n-1)$-sized subsets, or `patches', of a set of nodes.

\begin{definition}
\index{patch system}
Let $M = (\A,\G,\B)$ be an algebra-graph system. A \emph{patch system} for
$\B$ is a pair $\P = (P_1, P_2)$, where $P_1$ is a set and
$P_2 : [P_1]^{n-1} \to \B_+$
assigns an ultrafilter of $\B$ to each subset of $P_1$ of size $n-1$. 
(If $|P_1| < n-1$, then $P_2 = \emptyset$.)
A set $V = \{v_0, \dots, v_{n-1}\} \in [P_1]^n$ 
is said to be $\P$-\emph{coherent} if the following is satisfied: For
any $B_i \in P_2(V \setminus \{v_i\})$ ($i < n$), there are $p_i \in \G$
with $p_i \in B_i$ for each $i < n$, such that $\{p_0, \dots, p_{n-1}\}$ is not
an independent subset of $\G$.
The patch system $\P$ is said to be \emph{coherent} if every set
$V \subseteq P_1$ of size $n$ is $\P$-coherent.
\end{definition}

\begin{lemma} \label{L:coherent}
Let $M = (\A,\G,\B)$ be an algebra-graph system and $\P = (P_1, P_2)$ a
patch system for $\B$. Let $V = \{v_0, \dots, v_{n-1}\} \in [P_1]^n$ and for
each $i<n$, let $V_i = V \setminus \{v_i\}$. Then $V$ is $\P$-coherent if and
only if there exists an ultrafilter $\mu$ of $\A$ that is $i$-distinguishing
and with $\mu(i) = P_2(V_i)$ for each $i < n$.
\end{lemma}
\begin{proof}

$(\implies)$ Assume $V$ is $\P$-coherent. 
Define
\[
\mu_0 = \{S_i(B) \mid i<n,\;B \in P_2(V_i)\} \subseteq \A.
\]
To show that $\mu_0$ has the finite intersection property, it is sufficient to
consider arbitrary $B_i \in P_2(V_i)$ and prove that
$S_0(B_0) \cdot S_1(B_1) \cdots S_{n-1}(B_{n-1}) \neq 0$. By the $\P$-coherence of $V$, we can find
$p_i \in B_i$ for each $i < n$ such that $\{p_0, \dots, p_{n-1}\}$ is not an independent set. Now the following $\A$-universal sentence
holds in structures $M(\Gamma)$, because there is an atom 
$(K,\osim)$ that is
$i$-distinguishing and such that $K(i)=p_i$, for each $i<n$:
\begin{align*}
\forall B_0 \dots B_{n-1} \colon \B \bigg(\Big[ \exists p_0 \dots p_{n-1} \colon \G
                                   \Big({\bigwedge_{i<n} p_i \in B_i}
                                    \land \bigvee_{i<j < n}& E(p_i, p_j)\Big)\Big] 
                                    \\
                              & {} \rightarrow {\prod_{i<n}S_i(B_i)>0}\bigg).
\end{align*}
We showed that the left hand side of the implication is satisfied, so the right
hand side gives us that $\mu_0$ has the
finite intersection property. 
By the boolean prime ideal theorem, 
$\mu_0$ extends to an ultrafilter $\mu$ of $\A$. Since 
plainly $F_i = S_i(1^\B) \in \mu$, we have that
$\mu$ is $i$-distinguishing for all $i<n$. Moreover, if $B \in P_2(V_i)$, then
$S_i(B) \in\mu_0\subseteq \mu$, so by Lemma~\ref{L:RSProps}\ref{RSProps3},
$B = R_i(S_i(B)) \in \mu(i)$.
Therefore $P_2(V_i)  = \mu(i)$ by Lemma~\ref{L:ultraproj}\ref{ultraproj1},
since both sides are ultrafilters of $\B$.

\medskip

$({\impliedby})$ Assume $\mu$ is an ultrafilter of $\A$ that is $i$-distinguishing
for all $i < n$ and with $\mu(i) = P_2(V_i)$ for each $i < n$. Choose arbitrary
$B_i \in P_2(V_i)$ for each $i < n$. For each $i < n$, we can choose $b_i \in \mu$
such that $R_i(b_i) = B_i$. Let $b = \prod_{i<n} (b_i \cdot F_i) \in \mu$.
Now the following $\A$-universal sentence holds by definition in algebras from graphs,
because we can take $(K,\osim)\in x$, and then $\im K$ is not independent
and $K(i)\in R_i(x)$ for each $i$:
\[
\forall x \colon \A \bigg({0<x\leq\prod_{i<n} F_i} \rightarrow
                \exists p_0\ldots p_{n-1} \colon\G\Big({\bigwedge_{i<n}p_i\in R_i(x)}
                                 \land {\bigvee_{i<j<n}E(p_i,p_j)}\Big)\bigg).
\]
So we
can  choose $p_0, \dots, p_{n-1}$ with $p_i \in R_i(b)\subseteq R_i(b_i)=B_i$ and such that
$\{p_0, \dots, p_{n-1}\}$ is not independent. We conclude that $V$ is $\P$-coherent.
\end{proof}

\subsection{Patch systems from cylindric networks}
Here we show how to construct a coherent patch system
from a cylindric ultrafilter network.
We will need the following lemma to show that it is well defined.
We adopt the standard notation that
if $i,j<n$ then $[i/j]:n\to n$
denotes the function given by $[i/j](i)=j$ and $[i/j](k)=k$ for $k\neq i$.

\begin{lemma} \label{L:patchwd}
Let $M = (\A,\G,\B)$ be an algebra-graph system and $\N = (N_1, N_2)$ a
cylindric ultrafilter network over $\A$. 
Let $i,j<n$ and $v,w\in{}^nN_1$.
Then:
\begin{enumerate}
 \renewcommand{\theenumi}{(\roman{enumi})}
\renewcommand{\labelenumi}{(\roman{enumi})}

\item $N_2(v)$ is $i$-distinguishing
  if and only if $v$ is $i$-distinguishing.
  
  \item If $v$ is \ds i then $v\circ[i/j]$ is \ds j.
  
\item 
  $N_2(v)(i) = N_2(v\circ [i/j])(j)$.

\item If
  $\{v_k \mid i \neq k < n\} = \{w_k \mid j \neq k < n\}$ then
  $N_2(v)(i) = N_2(w)(j)$.

\end{enumerate}
\end{lemma}
\begin{proof}
(i) We have that $N_2(v) \ni F_i$ if and only if it does not contain $d_{jk}$ for
$j < k < n$ and $j,k \neq i$. But this is true if and only if $v$ is $i$-distinguishing
by the definition of cylindric ultrafilter networks.

\medskip

(ii) This is trivial if $i=j$, so suppose not.
If $k,l\neq j$ and $(v\circ[i/j])_k=(v\circ[i/j])_l$, then
$v_{[i/j](k)}=v_{[i/j](l)}$. Since the indices cannot be $i$,
this implies that ${[i/j](k)}={[i/j](l)}$.  As $k,l\neq j$, this implies that $k=l$
as required.

\medskip

(iii) Write $w = v\circ[i/j]$. Then
$w \equiv_i v$ and $w_i = v_j=w_j$. By the definition of ultrafilter network we
have $N_2(v) \equiv_i N_2(w)$ and $d_{ij} \in N_2(w)$. So by
Lemma~\ref{L:ultraproj}\ref{ultraproj4}  we have $N_2(v)(i) = N_2(w)(i)$, and by \ref{ultraproj2}  of the
same lemma, $N_2(w)(i) = N_2(w)(j)$.

\medskip

(iv) Assume the hypothesis.
Now $v$ is  \ds i  iff $|\{v_k \mid i \neq k < n\}|=n-1$,
and similarly for $w$.
So if $v$ is not \ds i then neither is $w$ \ds j,
and by part (i) and Lemma~\ref{L:ultraproj}\ref{ultraproj1},
$N_2(v)(i)=\B=N_2(w)(j)$ as required.
So assume that $v$ is \ds i, and hence that $w$ is \ds j.
We may suppose without loss of generality that $i = j = 0$ (by (ii,iii), we can just replace
$v$ by $v\circ[i/0]$ and $w$ by $w\circ[j/0]$).

The proof is by induction on the highest number $v, w$ disagree on:
$d(v,w) = \max\{k < n \mid v_k \neq w_k\}$. If they agree on everything or
$d(v,w) = 0$, then $v \equiv_0 w$, so $N_2(v) \equiv_0 N_2(w)$ and
Lemma~\ref{L:ultraproj}\ref{ultraproj4} gives us $N_2(v)(0) = N_2(w)(0)$.

Assume now that $d(v,w) = k > 0$ and the claim holds if $d(v,w)<k$. Since
$\{v_\ell \mid 0 \neq \ell < n\} = \{w_\ell \mid 0 \neq \ell < n\}$, $w_k = v_j$
for some $0 < j < n$. We have $j\neq k$ by definition of $k$.
If $j>k$, then $w_j = v_j = w_k$, contradicting that $w$ is $0$-distinguishing. So $j < k$. Now `swap' the $k$ and $j$
entries of $v$ --- that is, define
\[
v' = v\circ[0/k]\circ[k/j]\circ[j/0].
\]
By (ii), $N_2(v)(0) = N_2(v')(0)$. Also $v_k' = v_j = w_k$, and $v_\ell' = w_\ell=v_\ell$
for all $\ell > k$. So $v'$ is also $0$-distinguishing,
$\{v_\ell' \mid 0 \neq \ell < n\} = \{w_\ell \mid 0 \neq \ell < n\}$
and $d(v',w) < k$. So, using the induction hypothesis, we get
$N_2(v)(0) = N_2(v')(0) = N_2(w)(0)$.
\end{proof}

The third part in the above lemma says that the $i$th projection is
independent from the $i$th coordinate and the order of the elements in the vector.
This allows us to define the following:

\begin{definition}\label{def:DN}
Let $M = (\A,\G,\B)$ be an algebra-graph system and $\N = (N_1, N_2)$ a
cylindric ultrafilter network over $\A$. We define $\partial \N$ to be the patch system
$(N_1, P_2)$, where
\begin{align*}
P_2 : [N_1]^{n-1} &\to \B_+, \\
\{v_0, \dots, v_{i-1}, v_{i+1}, \dots, v_{n-1}\} &\mapsto N_2(v)(i),
\end{align*}
for each $i < n$ and $i$-distinguishing $v \in {}^nN_1$.
\end{definition}

\begin{proposition}\label{prop:DN coh}
Let $M = (\A,\G,\B)$ be an algebra-graph system and $\N=(N_1,N_2)$ a
cylindric ultrafilter network over $\A$. 
Then $\partial \N$ is a well defined and coherent patch system for $\B$.
\end{proposition}

\begin{proof}
Let $\partial \N=(N_1,P_2)$ as above.
By Lemma~\ref{L:patchwd}(iv),
 $P_2(\{v_0, \dots, v_{i-1},\allowbreak v_{i+1}, \dots,\allowbreak v_{n-1}\})\allowbreak=N_2(v)(i)$ is independent of the choice of $v,i$.
 By (i) of the lemma, $N_2(v)$ is \ds i,
 so by Lemma~\ref{L:ultraproj}\ref{ultraproj1}, $N_2(v)(i)\in\B_+$.
 So $\partial \N$ is well defined.
 Let $V=\{v_0,\ldots,v_{n-1}\}\in[N_1]^n$,
$V_i=V\setminus\{v_i\}\in [N_1]^{n-1}$
for $i<n$, and $\mu=N_2(v_0,\ldots,v_{n-1})$. By Lemma~\ref{L:patchwd}(i), $\mu$ is \ds i, 
and by definition of $\partial\N$, $P_2(V_i)=\mu(i)$, for every $i<n$.
By  Lemma~\ref{L:coherent}, $V$ is $\partial \N$-coherent.
As $V$ was arbitrary, $\partial \N$ is coherent.
\end{proof}

\subsection{Polyadic networks from patch systems}
A  patch system contains a lot of the 
information in an ultrafilter network.
Here we show that given a coherent patch system $\P=(P_1,P_2)$, we can always find ultrafilters to assign to
$n$-tuples of $P_1$ respecting $P_2$, and under fairly minimal conditions,
they form a polyadic ultrafilter network.

\begin{lemma}\label{lem:lots of ufs}
Let $M = (\A,\G,\B)$ be an algebra-graph system and  $\c P=(P_1,P_2)$ 
a coherent patch system for $\B$.
Let $v\in {}^nP_1$.
Then there is an ultrafilter $\mu$ of $\A$
such that
\begin{enumerate}
\item For $i,j < n$, we have $d_{ij} \in \mu$ if and only if $v_i = v_j$.

\item  $\mu(i)=P_2(\{v_j\mid j\in n\setminus\{i\}\})$ for each $i<n$ such that $v$ is \ds i.
\end{enumerate}
\end{lemma}

\begin{proof}
There are three cases.
\begin{itemize}
\item[(a)] If $|{\set(v)}| = n$, then by Lemma~\ref{L:coherent} there is an ultrafilter
  $\mu$ of $\A$ that is \ds i and with
  $\mu(i) = P_2(\{v_j \mid i \neq j < n\})$, for all $i < n$.

\item[(b)] If $|{\set(v)}|= n-1$, there are unique $i < j < n$ such that $v_i = v_j$,
and $v$ is \ds k iff $k\in\{i,j\}$.
By Lemma~\ref{L:ultraproj}\ref{ultraproj3},
  \[
\mu = \{a \in \A \mid R_i(a \cdot d_{ij}) \in P_2(\set(v))\} 
  \]
is an ultrafilter of $\A$
with $F_i,d_{ij}\in\mu$ (and hence $F_j\in\mu$ by Lemma~\ref{L:FiFj}),
so for each $k,l<n$ we have $d_{kl}\in\mu$ iff $v_k=v_l$.
Also,  $\mu(i)=P_2(\set v)$.
By Lemma~\ref{L:ultraproj}\ref{ultraproj2}, $\mu(j)=P_2(\set v)$ as well.

\item[(c)] If $|{\set(v)}| < n-1$, define
$
  D = \prod_{i < j < n, v_i = v_j} d_{ij} \cdot \prod_{i < j < n, v_i \neq v_j} -d_{ij}.
$
  By the generalisation technique, $D$ is an atom of $\A$ (in an algebra from a graph
  it would just be $(\emptyset, {\sim})$ where $i \sim j$ if and only if
  $v_i = v_j$). We define $\mu$ to be the principal ultrafilter of $\A$
  generated by~$D$. Condition~2 holds vacuously 
  as $v$ is never \ds i.

\end{itemize}\unskip
\end{proof}

\begin{lemma}\label{lem:nk from ps}
Let $M = (\A,\G,\B)$ be an algebra-graph system and  $\c P=(N_1,P)$ 
  a coherent patch system for $\B$.
  Suppose $N_2:{}^nN_1\to\A_+$ is a function satisfying the following,
  for any $v\in  {}^nN_1$:
\begin{enumerate}
\item For $i,j < n$, we have $d_{ij} \in N_2(v)$ if and only if $v_i = v_j$.
\item $N_2(v)(i)=P(\{v_j\mid j\in n\setminus\{i\}\})$  for each $i<n$ such that $v$ is \ds i.
\item If $\sigma:n\to n$ and $v\circ\sigma:n\to N_1$ is one-one, then
$N_2(v\circ\sigma)= N_2(v)^\sigma$.
\end{enumerate}
Then $(N_1,N_2)$  is a 
polyadic  ultrafilter network.\end{lemma}

\begin{proof} 
We check the  conditions from Definition~\ref{def:uf nk} defining ultrafilter networks.
The first condition, that
 $d_{ij} \in N_2(v)$ if and only if $v_i = v_j$ (for $v\in{}^nN_1$ and $i,j < n$), is given to us.
 It follows that $N_2(v)$ is \ds i iff $v$ is \ds i.
 
For the second condition, 
take  $i<n$ and $v,w\in {}^nN_1$ with $v \equiv_i w$. 
We require $N_2(v) \equiv_i N_2(w)$.
By assumption (2) of the lemma, if $v,w$ are \ds i we have
\[
N_2(v)(i) =  P_2(\{v_j \mid i \neq j < n\}) = P_2(\{w_j \mid i \neq j < n\})  =  N_2(w)(i),
\]
and if they are not, then by Lemma~\ref{L:ultraproj}\ref{ultraproj1}
we have $N_2(v)(i) =  \B  =  N_2(w)(i)$.
So by Lemma~\ref{L:ultraproj}\ref{ultraproj4}, $N_2(v) \equiv_i  N_2(w)$.

\medskip

Lastly we check the third condition for ultrafilter networks.
Let  $\sigma:n\to n$, let $w=v\circ\sigma$,
and let $\osim\in Eq(n)$ be given by $i\sim j$ iff $v_i=v_j$.
Observe that $i\sim^\sigma j$ iff $w_i=w_j$.
We check that
\[
N_2(w)=N_2(v)^\sigma.
\]
There are three cases.
If $|n/\osim^\sigma|=n$, then $v\circ\sigma$ is one-one and the result is given.

Suppose  that $|n/\osim^\sigma|=n-1$.
Let $\{i,j\}$ be the unique $\osim^\sigma$-class  of size 2.
By condition~1 of the lemma,  $F_i,d_{ij}\in N_2(w)$.
Also, if $k,l<n$ then 
$d_{kl}\in N_2(w)$ iff
$k\sim^\sigma l$, iff $v_{\sigma(k)}= v_{\sigma(l)}$, iff
 $s_\sigma d_{kl}=d_{\sigma(k)\sigma(l)}\in N_2(v)$ by 
Lemma~\ref{lem:sub tech}\ref{lem:sub tech2},
iff $d_{kl}\in N_2(v)^\sigma$.
Therefore, $F_i,d_{ij}\in N_2(v)^\sigma$ as well.
So by Lemma~\ref{L:ultraproj}\ref{ultraproj3},
it remains only to show that $N_2(w)(i)=(N_2(v)^\sigma)(i)$.

Now 
if $k,l\in n\setminus\{i\}$ and $\sigma(k)=\sigma(l)$, then certainly
$k\sim^\sigma l$, so  $k=l$ by assumption on $\sim^\sigma$.
Hence,
$\sigma$ is one-one on $n\setminus\{i\}$,
so
$\sigma[n\setminus\{i\}]=n\setminus\{l\}$ for some $l<n$.
We now obtain
\[
\begin{array}{rcll}
&&N_2(w)(i)
\\
&=&P_2(\{w_k\mid k\in n\setminus\{i\}\})&\mbox{by condition 2, since $w$ is \ds i}
\\
&=&P_2(\{v_{\sigma(k)}\mid k\in n\setminus\{i\}\})&\mbox{by definition of }w
\\
&=&P_2(\{v_k\mid k\in n\setminus\{l\}\})&\mbox{since }\sigma[n\setminus\{i\}]=n\setminus\{l\}
\\
&=&N_2(v)(l)&\mbox{by condition 2, since $v$ is  \ds l}
\\
&=&(N_2(v)^\sigma)(i)&\mbox{by Lemma~\ref{L:ultraproj}\ref{ultraproj6}}.
\end{array}
\]

Finally suppose that $|n/\osim^\sigma|<n-1$.
Then
$d_{\osim^\sigma}$ is an atom of $\A$
--- this is true in algebras from graphs,
because we have $d_{\osim^\sigma}=\{(\emptyset,\osim^\sigma)\}$,
so it holds for $M$ by the generalisation technique.
So $N_2(w)$ is the principal ultrafilter generated by $d_{\osim^\sigma}$.

Let $a\in N_2(v)^\sigma$ be arbitrary, so that $s_\sigma a\in N_2(v)$.
By the first part, $d_\osim\in N_2(v)$,
so $s_\sigma a\cdot d_\osim>0$.
By Lemma~\ref{lem:sub tech}\ref{lem:sub tech1} and \ref{lem:sub tech3},
$s_\sigma(a\cdot d_{\osim^\sigma})=s_\sigma a\cdot s_\sigma d_{\osim^\sigma}
\geq s_\sigma a\cdot d_\osim>0$,
so $a\cdot d_{\osim^\sigma}>0$ as well.
As $d_{\osim^\sigma}$ is an atom, we obtain $a\geq d_{\osim^\sigma}$ and $a\in N_2(w)$.
This shows that $N_2(v)^\sigma\subseteq N_2(w)$, and equality follows since 
by Lemma~\ref{lem:ss uf} both sides are ultrafilters of $\A$.
\end{proof}

\section{Chromatic number and representability}
\label{S:chromrep}

Here we show that the chromatic number of a graph $\Gamma$ and the  representability of
 $\A(\Gamma)$ and its reducts are tied together.

Recall that the chromatic number of a graph is
the size of the smallest partition into independent sets, or $\infty$ if no such
partition exists. Although the chromatic number is in general not first-order
definable, we can define an analogue for algebra-graph systems with the following formula.
\begin{definition}\label{def:theta}
\index{$\theta_k$}
For each $k < \omega$, we define the following $L_{AGS}$-sentence:
\[
\theta_k = \forall B_0 \dots B_{k-1} \colon\B \bigg( {\sum_{i < k} B_i = 1}
 \rightarrow \exists p, q \colon\G \big(E(p,q) \land \bigvee_{i < k} \left(p \in B_i \land q \in B_i\right) \big) \bigg),
\]
and $\Theta=\{\theta_k\mid k<\omega\}$.
\end{definition}
Then $M=(\A,\G,\B)\models\theta_k$ iff the chromatic number of $\G$ is larger than $k$
`as far as $\B$ can tell'. The true chromatic number of $\G$ may 
be smaller, but $\B$ contains no independent sets witnessing this.
However, $\B$'s estimate is correct when $\B=\wp(\G)$, as in structures of the form $M(\Gamma)$.
\begin{remark}
If  $M = (\A, \G, \B)$ is an algebra-graph system, we will say an element
$B \in \B$ is an independent set, if there are no $p,q \in B$  such that
$E(p,q)$.
\end{remark}

\subsection{Representable implies infinite chromatic number}
This direction can be proved without further help, apart from some of the
machinery from the preceding section and Ramsey's theorem.

\begin{proposition} \label{P:repimpchrom}
Let $M = (\A, \G, \B)$ be an  algebra-graph system in which $\B$ is infinite. If
the diagonal-free reduct of $\A$ is  representable,
then $M \models \Theta$.
\end{proposition}
\begin{proof}
Suppose for a contradiction that
the reduct of $\A$ to the signature of diagonal-free cylindric algebras
is  representable but
$M \not\models \theta_k$ for some $k<\omega$.

Recall (e.g., from \cite[\sec1.6]{Henkin71})
that for $a\in\A$, $\Delta a=\{i<n\mid c_ia\neq a\}$.
Define
 $D=\{a\in \A\mid \Delta a\neq n\}$,
 and let $\A'$ be the closure of $D$ under the boolean operations.
We first claim that $\A'$ is a subalgebra of $\A$. 
By Lemma~\ref{lem:A CA}, the cylindric reduct of $\A$ is a cylindric algebra.
By basic cylindric algebra, or the generalisation technique,
$\Delta0=\Delta1=\Delta d_{ii}=\emptyset$ for $i<n$; also,
$\Delta d_{ij}=\{i,j\}$ for distinct $i, j<n$, so since
$n\geq3$,
$\Delta d_{ij}\neq n$;
finally,
if $a\in\A'$ and $i<n$ then  $i\notin\Delta c_ia$.
So all these elements are in $D$ and hence in $\A'$.
Obviously,
$\A'$ is closed under $+$ and~$-$.
By Lemma~\ref{lem:sub tech}\ref{lem:sub tech5}, $D$ is  closed under
each $s_\sigma$, so by
Lemma~\ref{lem:sub tech}\ref{lem:sub tech1}, so is $\A'$.  This proves the claim.

Now let   $N=(\A',\G,\B)$.
We claim next that $N$ is a substructure of $M$.
Inspecting the function symbols of $L_{AGS}$, it suffices 
to show that  $S_i(B)\in\A'$ for every $B\in\B$ and $i<n$.
But by Lemma~\ref{L:RSProps}\ref{RSProps4},
$i\notin \Delta S_i(B)$,
so $S_i(B)\in D\subseteq\A'$.
This proves the claim.

As $M\models\U$ and all sentences in $\U$ are
 $\A$-universal, it follows that
$N\models \U$.
So $N$ is also 
an algebra-graph system in which $\B$ is infinite.
By Lemma~\ref{lem:A CA} and Corollary~\ref{C:simple}, the cylindric reduct
$\A'\restr{L_{\ca n}}$ is a simple cylindric algebra.
It is generated by $D$, and its diagonal-free reduct is representable
(since the diagonal-free reduct of $\A$ is).
It follows from a theorem of Johnson \cite[Theorem 1.8(i)]{Johnson69}
that $\A'\restr{L_{\ca n}}$ is representable as a cylindric algebra.
So by Lemma~\ref{L:easyembed}, there is a cylindric representation
$h$ that embeds $\A'\restr{L_{\ca n}}$ into a single cylindric set algebra
$\S = (\wp({}^nS), \cup, \setminus, \emptyset, {}^nS, D^S_{ij}, C^S_{i})_{i,j < n}$ with base set
$S$.

Let $\N$ be the ultrafilter network with nodes $S$ and
$\N( s) = \{a \in \A' \mid  s \in h(a)\} \in \A'_+$, for $s\in{}^nS$. This is easily seen to be
a well-defined cylindric ultrafilter network
over $\A'$.
Furthermore, by Proposition~\ref{prop:DN coh} we can
make it into a well-defined and coherent patch system $\partial \N$.

Now $M \not\models \theta_k$ means
that the following is true in $M$ and therefore $N$:
\[
 \exists B_0, \dots, B_{k-1} \colon \B \Big(
{\sum_{i<k} B_i = 1} \land \forall p,q 
\bigwedge_{i < k}\Big(
p \in B_i \land q \in B_i \rightarrow \lnot E(p,q) \Big)\Big).
\]
So $\G$ is the union of $k$
independent sets from $\B$: say, $B_0, \dots, B_{k-1}$.

Since $\B$ is infinite, by Lemma~\ref{L:RSProps}\ref{RSProps3} 
$\A'$ is also infinite.  As $h$ is injective, $\S$ is infinite and therefore $S$
as well. So we can choose infinitely many pairwise distinct elements $s_0, s_1, \dots$
from $S$. Now define a map $f : [\omega]^{n-1} \to k$ by letting
$f(\{i_1, \dots, i_{n-1}\})$ be the least $j < k$ such that
$B_j \in \partial \N (\{s_{i_1}, \dots, s_{i_{n-1}}\})$. By Ramsey's theorem
\cite{Ramsey30}, we
can choose the elements so that $f$ has constant value $c$, say.
Now consider   $\{s_0, \dots, s_{n-1}\}$.
Since $f$ is constant, $B_c \in \partial\N(\{s_j \mid i \neq j < n\})$ for all
$i < n$. Because $\partial\N$ is coherent, we can choose 
$p_0, \dots, p_{n-1} \in B_c$ so that $\{p_0, \dots, p_{n-1}\}$ is not an independent
set. But this is impossible since $B_c$ is independent.
\end{proof}

\subsection{Infinite chromatic number implies representable}
For the other direction, we define a game that allows us to build a
polyadic representation
for $\A$ if $M=(\A,\G,\B)\models\Theta$ 
(i.e., $\G$  has infinite chromatic number in the sense of $\B$).

\begin{definition}
\index{game}
Let $M = (\A, \G, \B)$ be an algebra-graph system.
A \emph{game} $G(\A)$ is an infinite sequence of polyadic ultrafilter networks
\[
\N_0 \subseteq \N_1 \subseteq \dots
\]
built by the following rules.
There are two players, named $\forall$ and $\exists$.
The game begins with the (unique) one-point network $\N_0$. There are $\omega$
rounds. In round $t < \omega$, the current network (at the start of the round)
 is $\N_t$ and  player
$\forall$ chooses an $n$-tuple $v \in {}^n\N_t$, a number $i < n$ and an element
$a \in \A$ such that $c_i a \in \N_t(v)$. The other player $\exists$ then has to
respond with an ultrafilter network $\N_{t+1} \supseteq \N_t$ such that there is
$w \in {}^n\N_{t+1}$ with $w \equiv_i v$ and $a \in \N_{t+1}(w)$. She wins the game
if she can play a network that satisfies these constraints in each round.
\end{definition}

\begin{lemma} \label{L:game}
Let $M = (\A, \G, \B)$ be an algebra-graph system. If $\exists$ has a
winning strategy in the game $G(\A)$, then $\A$ is a
representable polyadic equality algebra.
\end{lemma}
\begin{proof}
By the downward L\"owenheim--Skolem--Tarski theorem (see e.g. \cite{Chang90}), there is
a countable  elementary subalgebra $\A_0$ of $\A$. Let
$\N_0 \subseteq \N_1 \subseteq\cdots$ be a play of the game $G(\A)$ 
in which $\forall$ plays every possible
move in $\A_0$ and $\exists$ uses her winning strategy in $G(\A)$ to respond.
Define $\N = \bigcup_{t<\omega} \N_t$. This is certainly a
polyadic ultrafilter network over $\A$,
as all the $\N_t$ are polyadic ultrafilter networks.
Now define:
\begin{align*}
h : \A_0 &\to \big(\wp({}^n\N), \cup, \setminus, \emptyset, {}^n\N, D_{ij}^\N, C_i^\N,
S^\N_\sigma\mid i,j < n,\,\sigma:n\to n\big)\\
    a    &\mapsto \{v \in {}^n\N \mid a \in \N(v)\}.
\end{align*}
It can be checked that $h$ is 
a homomorphism.
Recall from Corollary~\ref{C:simple} that $\A_0$ is simple. So, since
$h(1) = {}^n\N \neq \emptyset = h(0)$, the map $h$ is injective. This shows
that $\A_0$ is representable, and because $\rpea{n}$ is a variety,
$\A$ is representable as well.
\end{proof}
\begin{remark}
The converse of the lemma also holds, but is not needed here.
\end{remark}

By the generalisation technique, in any 
algebra-graph system $(\A, \G, \B)$, $H$ defines an 
equivalence relation 
on $\G$ with $n$ classes, each  of which is in $\B$ since  the following
  $\A$-universal sentence is true in algebras from graphs:
  \[ 
  \forall x : \G\; \exists B : \B \; \forall y : \G (y \in B \leftrightarrow H(x,y)).
  \]

\begin{lemma} \label{L:ultraind}
Let $M = (\A, \G, \B)$ be an algebra-graph system such that
$M \models\Theta$.
Let $X$ be an equivalence class of $H$.
Then there is an ultrafilter $\nu$ 
of $\B$ that contains $X$ but contains no independent sets.
\end{lemma}
\begin{proof}
Let $\nu_0=\{B\in\B\mid X- B$ is independent\}. Then $\nu_0$ contains $X$ (clearly), and
has the finite intersection property: Suppose for a contradiction that for
$B_0, \dots, B_{k-1}\allowbreak \in \nu_0$ we have $B_0 \cdot B_1 \cdots B_{k-1} = 0$.
Then
\begin{align*}
X= X-(B_0 \cdot B_1  \cdots B_{k-1}) = (X- B_0) + (X- B_1) + \dots + (X- B_{k-1}).
\end{align*}
So $X$ is the union of $k$ independent sets in $\B$.
Now in any structure $M(\Gamma)$,
if an $H$-class is the union of $k$ independent sets in $\B$,
then copies of these sets  for every $H$-class lie in $\B$, 
so that $\Gamma$ is the union of $nk$ independent sets in $\B$
--- that is, $M(\Gamma)\models\neg\theta_{nk}$.
This implication  is $\A$-universal,
so it holds in $M$.
Hence, $M \not\models \theta_{nk}$, a contradiction.
Thus $\nu_0$ has the finite intersection property and, by the boolean prime ideal
theorem,  it can be extended to an ultrafilter $\nu$, which 
contains $X$ but no independent set (because it contains the complement).
\end{proof}
\begin{remark}
The converse of Lemma~\ref{L:ultraind} also holds, but is not needed here.
\end{remark}

\begin{proposition} \label{P:chromimprep}
Let $M = (\A, \G, \B)$ be an algebra-graph system. If
$M \models \Theta$, then $\A$ is representable
as a polyadic equality algebra.
\end{proposition}
\begin{proof}
By Lemma~\ref{L:game} it is sufficient to show that player $\exists$ has a winning strategy
in the game $G(\A)$. Suppose we are in round $t$ and the current polyadic ultrafilter network
is $\N_t$. According to the rules,  player $\forall$ chooses $a \in \A$, $i < n$
and $v \in {}^n\N_t$ with $c_ia\in \N_t(v)$. The other player $\exists$ now has to respond with a network
$\N_{t+1} \supseteq \N_t$ that contains some tuple 
$w \in {}^n\N_{t+1}$ such that $v \equiv_i w$
and $a \in \N_{t+1}(w)$. If there is already such  a  $w $ in $ {}^n\N_t$ 
then she can just respond with the unchanged network $\N_t$. So
we assume in the following that there is no such $w$.

\textbf{Step 1.} Let $N_{t+1} = \N_t \cup \{z\}$, where $z \not\in \N_t$ is a new
node. Let the tuple $w$ be defined by
$w \equiv_i v$ and $w_i = z$.
We will first try to find an ultrafilter of $\A$ for $w$.
 To help $\exists$ win the game, the ultrafilter
should contain $a$. We achieve this by showing that the following set has the
finite intersection property:
\[
\mu_0 = \{a\} \cup \{-d_{ij} \mid i \neq j < n\} \cup \{c_i b \mid b \in \N_t(v)\}.
\]

Let $D = \prod_{j \neq i} -d_{ij}$. We claim that
$c_i (a \cdot D) \in \N_t(v)$.
Assume for contradiction that
$c_i (a \cdot D) \not\in \N_t(v)$. 
Clearly, $D + \sum_{j \neq i} d_{ij} = 1$.
Therefore, $c_i a = c_i (a \cdot D) + \sum_{j \neq i} c_i (a \cdot d_{ij})\in\N_t(v)$.
So there is $j \neq i$ such that
$c_i (a \cdot d_{ij}) \in \N_t(v)$. 
Let $v' = v\circ[i/j]$. Then
$v \equiv_i v'$, so by  definition of ultrafilter networks,
$\N_t(v) \equiv_i \N_t(v')$. So 
$c_i (a \cdot d_{ij})=c_ic_i (a \cdot d_{ij}) \in \N_t(v')$ as well. But $v_i = v_j$,
 and therefore by  definition of ultrafilter networks, $d_{ij} \in \N_t(v')$. Thus
$d_{ij} \cdot c_i (a \cdot d_{ij}) \in N_t(v')$. In algebras from graphs
(and in cylindric algebras generally) we
certainly have
\[
\forall a : \A (d_{ij} \cdot c_i(a \cdot d_{ij}) \leq a).
\]
Hence, by the generalisation technique, $a \in N_t(v')$. 
But this contradicts our assumption
that no suitable tuple $w$ exists in ${}^n\N_t$.
So we must have
$c_i (a \cdot D) \in \N_t(v)$ as claimed.

Now, if $\mu_0$ failed the finite intersection property, there would
be  $b_0, \dots, b_{m-1} \allowbreak\in \N_t(v)$ such that
$a \cdot D \cdot c_i b_0 \cdots c_i b_{m-1} = 0$.
Then by cylindric algebra,
$0=c_i(a\cdot D\cdot c_i b_0 \cdots c_i b_{m-1})=c_i(a\cdot D)\cdot c_i b_0 \cdots c_i b_{m-1}\in \N_t(v)$, a contradiction.
Thus $\mu_0$ has
the finite intersection property.

By the boolean prime ideal theorem, player $\exists$ can
choose an ultrafilter $\mu$ of $\A$ that contains $\mu_0$. By construction,
 $\N_t(v)\equiv_i\mu$. Moreover,  
 \[ \label{E:dist}
d_{jk} \in \mu \iff w_j = w_k \tag{$\star$}
\]
for all $j,k < n$, because for $j\neq i$ we have $w_i \neq w_j$ and $-d_{ij} \in \mu$, and for
$j,k \neq i$,
\[
\begin{array}{ccccccc}
w_j = w_k 
&\Rightarrow 
& v_j = v_k 
&\Rightarrow
& d_{jk} \in \N_t(v)
& \Rightarrow 
& d_{jk} = c_i d_{jk} \in \mu,
\\
w_j \neq w_k 
&\Rightarrow
& v_j \neq v_k 
&\Rightarrow 
&-d_{jk} \in \N_t(v) 
&\Rightarrow
& -d_{jk} = c_i{-} d_{jk} \in \mu.
\end{array}
\]

\textbf{Step 2.} $\exists$ also needs to define ultrafilters for all the
remaining new tuples containing $z$. She can  do this with the help of the patch
system $\P = (N_{t+1}, P_2)$, defined as follows.  
\begin{itemize}
\item For each set of `old' nodes $V \in [\N_{t}]^{n-1}$, we define
  $P_2(V) = \partial \N_t(V)$.
\item For each $j < n$, define $W_j = \{w_k \mid j \neq k < n\}$. For each $W_j$
  of size $n-1$, she has to define $P_2(W_j)$.
  
For the case $j=i$, if $|W_i| = n - 1$ then because
  $W_i \subseteq \N_t$, she already defined $P_2(W_i) = \N_t(v)(i)= \mu(i)$
  (by Lemma~\ref{L:ultraproj}\ref{ultraproj4}).

  Now consider the $j \neq i$ with $|W_j| = n-1$.
Then $z\in W_j\not\subseteq \N_t$.
  We showed in \eqref{E:dist} that $\mu$ is $j$-distinguishing if $w$ is, so
  $\mu(j)$ is an ultrafilter of $\B$ in that case. So we define $P_2(W_j) = \mu(j)$.
  Note that this is well defined, because if there is $k \neq i,j$ such that
  $W_k = W_j$, then $w_j = w_k$, and thus by \eqref{E:dist} $d_{jk} \in \mu$
  and by Lemma~\ref{L:ultraproj}\ref{ultraproj2}, $\mu(j) = \mu(k)$.

\item For the remaining $W \in [N_{t+1}]^{n-1}$ that contain $z$, but that are
  not contained in $\set(w)$, we use a single ultrafilter constructed as follows. Recall
  that $H$ is an equivalence relation  on $\G$ with exactly $n$ equivalence
  classes, that satisfies the following for algebras from graphs:
  \[ \label{E:eqclasses}
  \forall x,y : \G (\lnot H(x,y) \rightarrow E(x,y)). \tag{$\dagger$}
  \]
  So by the generalisation technique, $(\dagger)$ is true for $H$ on $\G$.
  Call the equivalence classes $G_1, \dots G_n$.
Recall that they are contained in $\B$.  

  Now each of the $\mu(j)$ for $j \neq i$, if an ultrafilter of $\B$, contains exactly one of the $G_k$.
  There are at most $n-1$ such $j$,
  so there must be at least one $G_\ell$ that is not contained in any 
  $\mu(j)$ that is an ultrafilter. We are given that $M \models\Theta$, so by
  Lemma~\ref{L:ultraind} there is an ultrafilter $\nu$ of $\B$
  containing $G_\ell$ and no independent sets. We define $P_2(W) = \nu$ for
  all the remaining $W \in [N_{t+1}]^{n-1}$.
\end{itemize}
We check  that $\P$ is a coherent patch system.
Let $U = \{u_0, \dots, u_{n-1}\} \in [N_{t+1}]^n$ and write $U_j$ for
$U \setminus \{u_j\}$ for each $j < n$. We need to check that $U$ is
$\P$-coherent:
\begin{itemize}
\item If $z \not\in U$, then $U \subseteq \N_t$ and $U$ is $\P$-coherent because
$\N_t$ is a polyadic, hence cylindric network, so
by Proposition~\ref{prop:DN coh},
  $\partial\N_t$ is coherent.
\item If $U = \set(w)$, then $U$ is $\P$-coherent by Lemma~\ref{L:coherent}.

\item In the case where $z \in U$ and $|U \cap \set(w)| = n - 1$, we can find
  $j,k < n$ such that $z \in U_j = U \cap \set(w)$ and
  $z\in U_k \not\subseteq \set(w)$. Then, by the above, $G_\ell \in \nu=P_2(U_k)$. Moreover,
  by the choice of $\ell$, there is $m \neq \ell$ such that $G_m \in P_2(U_j)$.

  Take any $X_r \in P_2(U_r)$ for each $r < n$. Choose $p_r \in X_r$,
  for each $r<n$, with $p_j \in X_j \cdot G_m$ and $p_k \in X_k \cdot G_\ell$.
  Since $l \neq m$ and therefore $H(p_j,p_k)$ does not hold, we have $E(p_j, p_k)$
  by \eqref{E:eqclasses}. Thus $\{p_0, \dots, p_{n-1}\}$ is not independent.

\item In the remaining cases, $z \in U$ and $|U \cap \set(w)| < n - 1$. Then
  there are distinct $j, k < n$ such that $z \in U_j, U_k \not\subseteq \set(w)$. So by the above, we have
  $P_2(U_j) = P_2(U_k) = \nu$.

  Take any $X_r \in P_2(U_r)$ for each $r < n$. Then $X_j, X_k \in \nu$, and
  thus $X_j \cdot X_k \in \nu$ and is therefore not independent. So there are
  $p_j, p_k \in X_j \cdot X_k$ such that $E(p_j,p_k)$. For the other
  $s \neq j,k$ just choose any $p_s \in X_s$. Then
  $\{p_0, \dots, p_{n-1}\}$ is not independent.
\end{itemize}
This shows that $\P$ is coherent. 

We are nearly ready to define $\N_{t+1}$.
First, define an equivalence relation $\simeq$ on the set
of  one-one tuples in ${}^nN_{t+1}\setminus{}^n\N_t$, by:
$u\simeq u'$ iff there is a permutation $\sigma$ of $n$
such that $u\circ\sigma=u'$.
Choose a representative $u_\varepsilon$ of each $\simeq$-class $\varepsilon$,
ensuring that if $w$ is one-one then it is chosen as a representative.
We now define an ultrafilter $\N_{t+1}(u)$ of $\A$ for each $u\in {}^nN_{t+1}$
as follows.
\begin{enumerate}
\renewcommand{\theenumi}{U\arabic{enumi}}
\renewcommand{\labelenumi}{U\arabic{enumi}.}

\item\label{N case1} If $u\in{}^n\N_t$ we set $\N_{t+1}(u)=\N_t(u)$.

\item\label{N case2}  Define $\N_{t+1}(w)=\mu$.

\item\label{N case3}  If $u\in {}^nN_{t+1}\setminus({}^n\N_t\cup\{w\})$ is the 
representative of its $\simeq$-class
or is not one-one,
we use Lemma~\ref{lem:lots of ufs}
to choose any ultrafilter $\N_{t+1}(u)$
of $\A$ satisfying the properties of that lemma.

\item\label{N case4}  Each remaining tuple $u$ is one-one but is not the representative $u_\varepsilon$ of
its $\simeq$-class $\varepsilon$.
There is a unique $\sigma:n\to n$ such that $u=u_\varepsilon\circ\sigma$, and
we set $\N_{t+1}(u)=\N_{t+1}(u_\varepsilon)^\sigma$.
\end{enumerate}
We check that $\N_{t+1}$ is
a polyadic ultrafilter network.
It is sufficient to check that each $u\in{}^n\N_{t+1}$ satisfies the conditions
of Lemma~\ref{lem:nk from ps}, namely:
\begin{enumerate}
\renewcommand{\theenumi}{L\arabic{enumi}}
\renewcommand{\labelenumi}{L\arabic{enumi}.}
\item\label{lem rpt 1} For $j,k< n$, we have $d_{jk} \in \N_{t+1}(u)$ if and only if $u_j = u_k$.

\item\label{lem rpt 2} $\N_{t+1}(u)(j)=P_2(\{u_k\mid k\in n\setminus\{j\}\})$  for each $j<n$ such that $u$ is \ds j.

\item\label{lem rpt 3} If $\sigma:n\to n$ and $u\circ\sigma:n\to \N_{t+1}$ is one-one, then
$\N_{t+1}(u\circ\sigma)= \N_{t+1}(u)^\sigma$.
\end{enumerate}

If $u\in{}^n\N_t$ this is immediate because $\N_t$ is a polyadic
ultrafilter network and by definition of $\P$.
If $u=w$, \ref{lem rpt 1} holds by choice of $\mu$, 
\ref{lem rpt 2}  by definition of $\P$ and because $\mu\equiv_i\N_t(v)$,
and~\ref{lem rpt 3} by~\ref{N case4}  above, since $w$ is the representative of its $\simeq$-class.
If $u$ is not one-one then \ref{lem rpt 1} and \ref{lem rpt 2} hold by choice of $\N_{t+1}(u)$
in \ref{N case3},
and \ref{lem rpt 3} holds vacuously.
All that remains is the case where $u\notin {}^n\N_t\cup\{w\}$ is one-one.
Let $\varepsilon$ be the $\simeq$-class of $u$, and let
$u=u_\varepsilon\circ\tau$ for some (unique) $\tau:n\to n$.
Trivially if $u=u_\varepsilon$, and by~\ref{N case4}
otherwise,  $\N_{t+1}(u)=\N_{t+1}(u_\varepsilon)^\tau$.
Below, $j,k$ range over $n$.
\begin{itemize}
\item For \ref{lem rpt 1},  $d_{jk}\in \N_{t+1}(u)=\N_{t+1}(u_\varepsilon)^\tau$
iff $s_\tau d_{jk}=d_{\tau(j)\tau(k)}\in \N_{t+1}(u_\varepsilon)$
by Lemma~\ref{lem:sub tech}\ref{lem:sub tech2},
iff $(u_\varepsilon)_{\tau(j)}=(u_\varepsilon)_{\tau(k)}$ 
by choice of $\N_{t+1}(u_\varepsilon)$,
iff $u_j=u_k$ as required.

\item 
We check \ref{lem rpt 2}.
Suppose that $u$ is \ds j.
Plainly, $\tau$ is one-one, so $\tau[n\setminus\{j\}]=n\setminus\{\tau(j)\}$.
Consequently,
\[
\begin{array}{rcll}
&&\N_{t+1}(u)(j)
\\
 &=& \N_{t+1}(u_\varepsilon)^\tau(j) &\mbox{by definition of }\N_{t+1}(u)
\mbox{ in \ref{N case4}}
\\
&=&\N_{t+1}(u_\varepsilon)(\tau(j)) &\mbox{by Lemma~\ref{L:ultraproj}\ref{ultraproj6}}
\\
&=&P_2(\{(u_\varepsilon)_{k}\mid k\in n\setminus\{\tau(j)\}\})&
\mbox{by choice of }\N_{t+1}(u_\varepsilon)
\\
&=&P_2(\{(u_\varepsilon)_{\tau(k)}\mid k\in n\setminus\{j\}\})&
\mbox{as }n\setminus\{\tau(j)\}=\tau[n\setminus\{j\}]
\\
&=&P_2(\{u_k\mid k\in n\setminus\{j\}\})&
\mbox{as }u=u_\varepsilon\circ\tau.
\end{array}
\]

\item For~\ref{lem rpt 3}, suppose that $ \sigma:n\to n$ and $u\circ\sigma$ is one-one.
We check that   $\N_{t+1}(u\circ\sigma)=\N_{t+1}(u)^\sigma$.
Plainly, 
 $u\circ\sigma=u_\varepsilon\circ\tau\circ\sigma\in\varepsilon$
and $\tau\circ\sigma$ is one-one.
Using the definitions and Lemma~\ref{lem:ss uf},
\begin{align*}
\N_{t+1}(u\circ\sigma)=\N_{t+1}(u_\varepsilon)^{\tau\circ\sigma}
= (\N_{t+1}(u_\varepsilon)^\tau)^\sigma =\N_{t+1}(u)^\sigma,
\end{align*}
as required.
\end{itemize}
So by Lemma~\ref{lem:nk from ps}, $\N_{t+1}$ is a polyadic
ultrafilter network.
We also have $\N_{t+1} \supseteq \N_t$, $w \equiv_i v$, and $a \in \mu = \N_{t+1}(w)$. 
The network $\N_{t+1}$ is  $\exists$'s response to $\forall$'s move in round~$t$.
So she is able to respond
to any move made by $\forall$ --- she has a winning strategy.
\end{proof}

\begin{definition}\label{def:thys}
Let us define some $L_{AGS}$-theories.
\begin{enumerate}
\item Fix a universal axiomatisation
$\Pi$ of $\rpea n$ --- such an axiomatisation  exists because
 $\rpea n$ is a variety (Proposition~\ref{prop:can}). 
Also fix any first-order axiomatisation 
$\Delta$ of $\rdf n$.
We regard 
$\Pi$ and $\Delta$ as $\A$-sorted 
$L_{AGS}$-theories in the obvious way.

\item 
Let $\Phi$ be the following $L_{AGS}$-theory, expressing that $\B$ is infinite:
\[
\Phi = \{\phi_m \mid m < \omega\} \quad \text{where }
  \phi_m = \exists B_0, \dots,B_{m-1} : \B \big(\bigwedge_{i < j < m} B_i \neq B_j\big).
\]

\item Also recall from Definition~\ref{def:theta}
that $\Theta=\{\theta_k\mid k<\omega\}$ expresses that $\G$ has infinite chromatic number in the $\B$-sense. The theory $\U$ defining algebra-graph systems
was laid down in Definition~\ref{D:ags}.

\end{enumerate}
\end{definition}
We now obtain the main result of this section.
It generalises the analogous result for algebras from graphs  in \cite{Hirsch09}.

\begin{theorem}\label{thm:dragalong}
 $\U\cup\Phi\cup\Delta\models\Theta$,
and  $\U\cup\Theta\models\Pi$.
\end{theorem}

\begin{proof}
Immediate from Propositions~\ref{P:repimpchrom} and~\ref{P:chromimprep}.
\end{proof}

\section{Applications}
\label{S:dirinvsys}

Here we apply Theorem~\ref{thm:dragalong}
to prove our two main theorems.

\begin{definition}
For $L_{\df n}\subseteq L\subseteq L_{\pea n}$,
we write $\rl$ for the class of
$L$-algebras having a representation respecting all the $L$-operations.
\end{definition}

\subsection{Strongly representable atom structures}

\begin{definition}
Let $L_{\df n}\subseteq L\subseteq L_{\pea n}$.
An $L$-atom structure $\S$ is said to be \emph{strongly representable}
if $\S^+\in\rl$.
\end{definition}

The following generalises the main result of \cite{Hirsch09}
to other signatures.
It has already been proved  by Sahed Ahmed
 (draft of untitled monograph, 2010) using the same algebras.
\begin{theorem}\label{thm:srb nonelem}
For any $L_{\df n}\subseteq L\subseteq L_{\pea n}$,
the class of strongly representable $L$-atom structures is non-elementary.
In another common notation, the class
${\sf Str}\,\rl$ of structures for $\rl$ is non-elementary.
\end{theorem}

\begin{proof}
A celebrated result of Erd\H os \cite{Erdos59} 
shows that for all $k<\omega$ there is a finite graph $G_k$
with chromatic number and girth (length of the shortest cycle)
both at least $k$.
Let $\Gamma_k$ be the disjoint union of the $G_\ell$ for $k\leq \ell<\omega$:
this time, no edges are added between copies.
Plainly, $\Gamma_k$ has infinite chromatic number, and its girth is at least $k$.
By Proposition~\ref{P:chromimprep} applied to $M(\Gamma_k)$,
$\A(\Gamma_k)\restr L\in\rl$, so that
$(\At\Gamma_k)\restr L_+$ is strongly representable.

Now let $\Gamma$ be a non-principal ultraproduct of the $\Gamma_k$.
Then $\Gamma$ is infinite, and by \L o\'s's theorem it has girth at least $k$ for all finite $k$,
since this property is first-order definable.
Hence, $\Gamma$ has no cycles, so its chromatic number is at most two.
By Proposition~\ref{P:repimpchrom},
the diagonal-free reduct of $\A(\Gamma)$ is not representable,
and hence neither is its $L$-reduct.
So $(\At\Gamma)\restr L_+$ is not strongly representable.

But it is easily seen that the operation $\At(-)$ commutes with ultraproducts,
and it follows that $(\At\Gamma)\restr L_+$ is isomorphic to an ultraproduct
of the $(\At\Gamma_k)\restr L_+$.
This shows that the class of strongly representable $L$-atom structures
is not closed under ultraproducts and so cannot be elementary.
\end{proof}

\subsection{Canonical axiomatisations}
Here, we use direct and inverse systems to build a certain algebra,
and apply the results from the previous sections to show that it can be made to satisfy an arbitrary number of representability axioms,
while its canonical extension only satisfies a bounded number.
It will follow that any first-order
axiomatisation of the representable cylindric algebras (and various other classes)
has infinitely many non-canonical axioms.

Our argument is based on the following result.
It is from \cite[Lemma 4.1]{Hodkinson05}, 
but it can be proved in a rather simpler way by modifying the 
argument of \cite[Theorem 4]{Hell92}. Both proofs use similar random graphs.
First, a definition. 

\begin{definition}
Let $\Gamma,\Delta$ be graphs.
A map $f\colon \Gamma\to\Delta$ is
said to be a \emph{graph p-morphism} if
for each $x\in\Gamma$, $f$ maps the set of neighbours of $x$ in $\Gamma$
surjectively onto the set of neighbours of $f(x)$ in $\Delta$. 
\end{definition}

\begin{theorem}\label{thm:erdos}
Suppose that $2 \leq \ell \leq k < \omega$.
Then there exists an inverse system of  finite graphs 
\[
\begin{array}{ccccc}
  \Gamma_0 &    \xtwoheadleftarrow{f_{10}}& \Gamma_1
  & \xtwoheadleftarrow{f_{21}}& \cdots,
\end{array}
\]
where the $f_{ij}$ are surjective graph p-morphisms, such that
 $\chi(\Gamma_s) = k$
for every  $s<\omega$,  and $\chi(\varprojlim \Gamma_s) = \ell$.
\end{theorem}

Our algebras are constructed from atom structures based on graphs,
so we need to transform graph p-morphisms into 
p-morphisms of atom structures, and then, using duality, to embeddings of algebras.
We will also consider direct and inverse
systems, and their limits.

\begin{definition}\label{def:pmorph}
Let $L\supseteq L_{BA}$ be a functional signature and let
$\S=(S,R_f\mid f\in L\setminus L_{BA})$
and $\S'=(S',R'_f\mid f\in L\setminus L_{BA})$ be $L$-atom structures.
Let $g:S\to S'$. be a function.
We say that $g:\S\to\S'$ is a \emph{p-morphism of atom structures}
if for each $n$-ary $f\in L\setminus L_{BA}$, we have:
\begin{description}
\item [Forth:] $g$ is an $L_+$-homomorphism:
for every $\vec xn,y\in S$, if $R_f(\vec xn,\allowbreak y)$
then $R'_f(g(x_1),\allowbreak\ldots,\allowbreak g(x_n), \allowbreak g(y))$.

\item [Back:] if $y\in S$, $\vec {x'}n\in S'$, and $R'_f(\vec {x'}n,g(y))$,
then there are $\vec xn\in S$ such that
$R_f(\vec xn,y)$ and $g(x_i)=x'_i$ for $i=1,\ldots,n$.

\end{description}
\end{definition}
Our first lemma is straightforward.

\begin{lemma} \label{L:pmorph}
Let $\Gamma,\Delta$ be graphs and $f:\Gamma\to\Delta$ a surjective graph p-morphism.
Let $f^\times:\Gamma\times n\to\Delta\times n$
be given by $f^\times(p,i)=(f(p),i)$ for $(p,i)\in\Gamma\times n$.
Define
  \begin{align*}
    \widehat f : \At(\Gamma) &\to \At(\Delta),\hskip36pt
    (K, \sim) \mapsto (f^\times\circ K, \sim).
  \end{align*}
Then $\widehat f$ is a surjective p-morphism  of atom structures.
\end{lemma}
\begin{proof}
Plainly, $f^\times:\Gamma\times n\to\Delta\times n$
is a surjective graph p-morphism.
We need to check the following:
  \begin{enumerate}
   \renewcommand{\theenumi}{(\roman{enumi})}
\renewcommand{\labelenumi}{(\roman{enumi})}

  \item  if $(K, \sim) \in \At(\Gamma)$, then
    $\widehat f(K, \sim) \in \At(\Delta)$;
    
  \item surjectivity;
  \item
  the forth property of the cylindrification relations, i.e. if we
    have $i < n$ and $(K^1, \sim^1) \equiv_i (K^2, \sim^2)$ then
    $\widehat f (K^1, \sim^1) \equiv_i \widehat f (K^2, \sim^2)$;

 \item the back property of the cylindrification relations, i.e. if we
    have $i < n$ and $(J^2, \sim^2)\equiv_i \widehat f (K^1, \sim^1)$, then
    there is $(K^2, \sim^2)\in\At(\Gamma)$ such that $\widehat f(K^2, \sim^2) = (J^2, \sim^2)$
    and $(K^2, \sim^2)\equiv_i(K^1, \sim^1)$;
  \item diagonals are preserved, i.e.
    $(K, \sim) \in D_{ij} \iff \widehat f(K, \sim) \in D_{ij}$;

  \item   substitutions are preserved:
  $\widehat f((K,\osim)^\sigma)=(\widehat f(K,\osim))^\sigma$.

  \end{enumerate}

For (i), suppose  $(K, \sim) \in \At(\Gamma)$ and $|n/{\sim}| = n$.
   Clearly the domain of $K$ is preserved by $\widehat f$. Moreover,
   since $\im K$ is not independent and $f^\times$
   is a graph p-morphism, $\im K'$ is not independent either. The other cases follow
   directly from the definition of $\widehat f$.
  
  To show (ii) let $(K', \sim) \in \At(\Delta)$. If $K'$ is not defined
  anywhere, we let $K$ be undefined everywhere as well. If there are
  $i < j < n$ such that $i \sim j$ and $K'(i) = K'(j)$ is defined, 
  then as $f^\times$ is surjective, there is
  $p \in \Gamma \times n$ such that $f^\times (p) = K'(i)$. Define
  $K(i) = K(j) = p$ and let $K$ be undefined for the remaining values
  in that case. Finally, if $K'$ is defined on all values $i < n$, then
  $\im(K')$ is not independent, so there are $i < j < n$ such that there
  is an edge from $K'(i)$ to $K'(j)$. Since $f^\times$ is surjective, 
  there is  $p_i \in \Gamma \times n$ such that
  $f^\times(p_i) = K'(i)$. 
As  $f^\times$ is a graph p-morphism,  there is
  $p_j \in \Gamma \times n$ such that there is an edge between $p_j$
  and $p_i$ and  $f^\times(p_j) = K'(j)$. For the remaining $s \neq i,j$,
using surjectivity   we take any vertices $p_s \in \Gamma \times n$ such that
  $f^\times (p_s) = K'(s)$. Now define $K(s) = p_s$ for each $s<n$. By construction,
  $(K, \sim) \in \At(\Gamma)$ in all three cases, and
  $\widehat f (K, \sim) = (K', \sim)$.

  For (iii) we have for $(K^1, \sim^1), (K^2, \sim^2) \in \At(\Gamma)$
  and $i < n$ that
  \begin{align*}
    &(K^1, \sim^1) \equiv_i (K^2, \sim^2) \\
    \implies &K^1(i) = K^2(i) \text{ and }
    {\sim_i^1} = {\sim_i^2} \\
    \implies &f^\times(K^1(i)) = f^\times(K^2(i)) \text{ and }
    {\sim_i^1} = {\sim_i^2} \\
    \implies &\widehat f(K^1, \sim^1)
    \equiv_i \widehat f(K^2, \sim^2).
  \end{align*}
  
  For (iv), suppose that
    $(K^1, \sim^1) \in \At(\Gamma)$, $(J^2, \sim^2) \in \At(\Delta)$, $i < n$,
    and $\widehat f(K^1, \sim^1) \equiv_i (J^2, \sim^2)$.
    Then
  \begin{align*}
    f^\times(K^1(i)) = J^2(i) \text{ and }
     {\sim_i^1} = {\sim_i^2}.
   \end{align*}
   Now take $(K^2, \sim^2)$ such that $K^2(i) = K^1(i)$ (which may be undefined), and if $j \neq i$, we
   choose $K^2(j)$ from the $f^\times$-pre-image of $J^2(j)$ if $J^2$ is defined for
   $j$,  and otherwise we leave $K^2(j)$ undefined. 
It is not hard to do this in such a way that if $j\sim^2 k$ then $K^2(j)=K^2(k)$,
   and if $K^2$ is total then $\im K^2$ is not independent
   (here we use that $\im J^2$ is not independent and $f^\times$ is a graph p-morphism).
   Then $(K^2,\sim^2)\in \At(\Gamma)$,
   $\widehat f(K^2, \sim^2) = (J^2, \sim^2)$, and $(K^1, \sim^1) \equiv_i (K^2, \sim^2)$.
   
   To see that diagonals are preserved (v), note that
$
   (K, \sim) \in D_{ij} \iff i \sim j
   \iff \widehat f (K, \sim) \in D_{ij}.
$   
   For (vi), we have 
 \begin{align*}
\widehat f((K,\osim)^\sigma)&=(f^\times\circ K^\sigma,\osim^\sigma),
\\
(\widehat f(K,\osim))^\sigma&=((f^\times\circ K)^\sigma,\osim^\sigma).
\end{align*}
Recall that in general, $K^\sigma(i)$ is defined iff $\osim^\sigma$ is \ds i and is then
$K(j)$, where $j\notin\sigma[n\setminus\{i\}]$.
So $f^\times\circ K^\sigma(i)$ is defined iff $(f^\times\circ K)^\sigma(i)$ is defined,
and in that case, 
\[
f^\times\circ K^\sigma(i)=f^\times(K(j))=(f^\times\circ K)(j)=
(f^\times\circ K)^\sigma(i).
\]
So indeed, $\widehat f((K,\osim)^\sigma)=(\widehat f(K,\osim))^\sigma$.
 \end{proof}

\begin{lemma}\label{lem:dool 2}
Let $g:\At(\Gamma)\to\At(\Delta)$ be a surjective p-morphism.
Then the  map
\[
g^+:\A(\Delta)\to\A(\Gamma),\quad Y\mapsto \{x\in\At(\Gamma)\mid g(x)\in Y\}
\]
is an algebra embedding.
If $f:\A(\Delta)\to\A(\Gamma)$ is an embedding, then the map
\[
f_+:\A(\Gamma)_+\to\A(\Delta)_+,\quad \mu\mapsto \{a\in\A(\Delta)\mid f(a)\in \mu\}
\]
is a surjective p-morphism.
\end{lemma}

\begin{proof}
This is standard duality: see, e.g.,  \cite[theorem 5.47]{Blackburn01}.
\end{proof}

\begin{proposition}\label{prop:invsys}
Let $\mathfrak G=(\Gamma_n,\nu^m_n:n\leq m<\omega)$ be an inverse system
of finite graphs and surjective p-morphisms.
In the notation of Lemmas~\ref{L:pmorph} and~\ref{lem:dool 2}, define 
\[
\begin{array}{rcl}
\At(\mathfrak G)&=&(\At(\Gamma_n),\widehat{\nu^m_n}:n\leq m<\omega),
\\[2pt]
\c A(\mathfrak G)&=&(\c A(\Gamma_n),\widehat{\nu^m_n}^+:n\leq m<\omega),
\\[2pt]
\c A(\mathfrak G)_+&=&(\c A(\Gamma_n)_+,(\widehat{\nu^m_n}^+)_+:n\leq m<\omega).
\end{array}
\]
Then:
\begin{enumerate}
   \renewcommand{\theenumi}{(\roman{enumi})}
\renewcommand{\labelenumi}{(\roman{enumi})}
\item $\At(\mathfrak G)$ is an inverse system of atom structures and
surjective p-morph\-isms,\label{invsys}

\item $\c A(\mathfrak G)$ is a direct system of BAOs and embeddings,\label{embed}

\item $\A(\mathfrak G)_+$ is an inverse system of atom structures and
surjective p-morph\-isms, and $\A(\mathfrak G)_+\cong\At(\mathfrak G)$,\label{invsys2}

\item $\big(\displaystyle\varinjlim\c A(\mathfrak G)\big)_+\cong\varprojlim\big(\A(\mathfrak G)_+\big)$,\label{rob's bit}

\item $(\displaystyle\varinjlim\c A(\mathfrak G))_+\cong\At(\varprojlim\mathfrak G)$.\label{rob's bit 2}

\end{enumerate}
\end{proposition}

\begin{proof}
Parts \ref{invsys}--\ref{invsys2} are almost immediate from Lemmas~\ref{L:pmorph} and~\ref{lem:dool 2}.
For the last item in~\ref{invsys2},  as each $\At(\Gamma_n)$ is finite, 
 $\c A(\Gamma_n)_+\cong\At(\Gamma_n)$
 (see, e.g., \cite[theorems 9.2, 10.7]{Goldblatt76}),
 and this can be easily extended to show that
  $\c A(\mathfrak G)_+\cong\At(\mathfrak G)$.
  
Part~\ref{rob's bit} is a consequence of important results
of Goldblatt \cite[theorems~10.7, 11.2, 11.6]{Goldblatt76}.
Goldblatt proved these results for 
modal algebras, but they  generalise easily to BAOs.
  
For part~\ref{rob's bit 2},
by part~\ref{rob's bit} and~\ref{invsys2} we have
\[
(\varinjlim\c A(\mathfrak G))_+\cong\varprojlim(\c A(\mathfrak G)_+)\cong\varprojlim\At(\mathfrak G).
\] 
It is clear  that $\At(-)$ commutes with inverse limits, so that
$\varprojlim\At(\mathfrak G)\cong\At(\varprojlim\mathfrak G)$.
\end{proof}

We can now prove the main result of the paper.

\begin{theorem} \label{T:main}
Let $L$ be a signature satisfying $L_{\df n}\subseteq L\subseteq L_{\pea n}$.
Then any first-order axiomatisation of $\rl $ contains infinitely
many non-canonical axioms.
\end{theorem}
\begin{proof}
Suppose for a contradiction that $T=T_C\cup T_{NC}$ is a first-order axiomatisation of $\rl$,
where every sentence in $T_C$ is canonical and  $T_{NC}$ is finite.
We regard $T$ equally as an $\A$-sorted $L_{AGS}$-theory in the natural way.
Plainly, $\Pi\models T\models\Delta$.
Also, by Theorem~\ref{thm:dragalong}, 
 $\U\cup\Phi\cup\Delta\models\Theta$
and  $\U\cup\Theta\models\Pi$.
Using this and first-order compactness, and bearing in mind that
$\theta_k\models\theta_l$ whenever $l\leq k<\omega$,
we see that:
\begin{enumerate}
\item there is  $\ell<\omega$ such that
$\U\cup\{\theta_\ell\}\models T_{NC}$,

\item  there is a finite $T_0\subseteq  T_{C}$
such that
$\U\cup\Phi\cup T_0\cup T_{NC}\models\theta_{\ell+1}$,

\item  there is a finite $\Pi_0\subseteq\Pi$
such that $\Pi_0\models T_0$, 

\item there is $k<\omega$
such that $k>\ell$ and $\U\cup\{\theta_k\}\models \Pi_0$. 
\end{enumerate}

Using Theorem~\ref{thm:erdos},
 take finite graphs $\Gamma_0, \Gamma_1, \dots$ such that
$\chi(\Gamma_s) = k+1$ for all $s < \omega$,
\[
\begin{array}{ccccc}
  \Gamma_0 &    \xtwoheadleftarrow{f_{10}} & \Gamma_1
  & \xtwoheadleftarrow{f_{21}} & \cdots,
\end{array}
\]
where the $f_{ij}$ are surjective graph p-morphisms,
and, writing $\Gamma=\varprojlim \Gamma_s$, we have
$\chi(\Gamma) = \ell+1$.
Using Proposition~\ref{prop:invsys}\ref{embed}, we obtain
embeddings:
\[
\A(\Gamma_0) \hookrightarrow   \A(\Gamma_1) \hookrightarrow  \dots.
\]
Define $\A = \varinjlim \A(\Gamma_s)$. Then, because
$\chi(\Gamma_s) = k+1$, we have $M(\Gamma_s)\models\U\cup\{\theta_k\}$,
so
$\A(\Gamma_s) \models \Pi_0$ for each $s<\omega$. As the sentences in $\Pi$
are universal, they are preserved by direct limits, and we therefore
have $\A \models \Pi_0$ and hence $\A\models T_0$. 
As all sentences in $T_0$ are canonical, $\A^\sigma\models T_0$ as well.
Moreover, from Proposition~\ref{prop:invsys}\ref{rob's bit 2} we get
\[
\A_+=\big(\varinjlim \A(\Gamma_s)\big)_+\cong  At(\varprojlim \Gamma_s)=At(\Gamma),
\]
and thus $ \A^\sigma\cong\A(\Gamma)$
and $M(\Gamma) \cong (\A^\sigma,\Gamma,\wp(\Gamma))$.
We chose the graphs so that $\chi(\Gamma) =\ell+1$.
So $M(\Gamma)\models\U\cup\{\theta_\ell\}$ and hence
$\A^\sigma \models T_{NC}$.
As $\Gamma$ is plainly infinite, $\wp(\Gamma)$ is also infinite, and so
   $M(\Gamma)\models   \U\cup\Phi\cup T_0\cup T_{NC}$
   and hence $M(\Gamma)\models \theta_{\ell+1}$.
   So $\chi(\Gamma) > \ell+1$, a contradiction.
\end{proof}

\begin{corollary}\label{cor:nfa etc}
Any first-order axiomatisation (for example, any equational axiomatisation) of any of the following
classes has infinitely many non-canonical sentences:
\begin{enumerate}
\item the class $\rdf n$ of representable $n$-dimensional 
diagonal-free cylindric algebras,

\item the class $\rca n$ of representable $n$-dimensional cylindric algebras, 

\item the class $\rpa n$ of representable $n$-dimensional 
polyadic algebras,

\item the class $\rpea n$ of representable $n$-dimensional 
polyadic equality algebras.
\end{enumerate}
Hence, none of the classes is finitely axiomatisable, nor does it have
an axiomatisation where only finitely many axioms are not Sahlqvist equations.
\end{corollary}
\begin{proof}
Immediate from  Theorem~\ref{T:main} and because Sahlqvist equations are canonical.
\end{proof}

\section{Conclusion}\label{sec:end}
We have proved that every variety of representable
algebras of relations whose signature lies between that of $\rdf n$ and $\rpea n$
(for finite $n\geq3$) is \emph{barely canonical,}
in that (although canonical) it cannot be axiomatised by first-order sentences
only finitely many of which are not themselves canonical.
As far as we know, it is an open question whether various
other varieties of algebras of relations are also barely canonical, including
infinite-dimensional diagonal-free, cylindric,
polyadic (equality) and quasi-polyadic (equality) algebras,
classes of relativised set algebras such as ${\sf Crs}_n$, ${\sf D}_n$, ${\sf G}_n$
($n\geq3$),
and various classes of neat reducts,
such as ${\bf S}\mathfrak{Nr}_n\ca m$ for $3\leq n<m<\omega$, 
and ${\bf S}\mathfrak{Ra}\ca n$ for $5\leq n<\omega$.
Some of these (such as ${\sf G}_\omega$) are not even known to be varieties.
A wider question is to find a more general method for proving bare canonicity.

\bibliographystyle{amsplain}

\bibliography{bib}

\end{document}